\documentclass{daj}

\usepackage{amsthm,amssymb,amsmath}
\usepackage{stmaryrd}
\usepackage{tikz}
\usepackage[shortlabels]{enumitem}

\numberwithin{equation}{section}
\setlist{leftmargin=3\parindent,labelindent=3\parindent}
\setlist[enumerate]{%
  leftmargin=3\parindent,%
  align=left,%
  labelwidth=3\parindent,%
  labelsep=0pt%
}
\setlist[enumerate,1]{%
  label={\normalfont (\thesection.\arabic{equation})}, ref={\normalfont \thesection.\arabic{equation}},
  resume%
}
\allowdisplaybreaks

\newtheorem{thm}[equation]{Theorem}
\newtheorem{cor}[equation]{Corollary}
\newtheorem{lem}[equation]{Lemma}
\newtheorem{prop}[equation]{Proposition}
\newtheorem{conj}[equation]{Conjecture}
\newtheorem{prob}[equation]{Problem}
\newtheorem{obs}[equation]{Observation}
\theoremstyle{definition}
\newtheorem{defn}[equation]{Definition}
\newtheorem{ex}[equation]{Example}

\newcommand{\perm}[1]{{\small \mbox{\tt{#1}}}}
\newcommand{\R}{\mathbb{R}}
\newcommand{\shadebox}[3]{\draw[black!#3!white,fill=black!#3!white] (#2,#1)--(#2-1,#1)--(#2-1,#1-1)--(#2,#1-1);}

\DeclareTextCompositeCommand{\v}{OT1}{l}{l\nobreak\hspace{-.1em}'}

\dajAUTHORdetails{%
  title = {Six Permutation Patterns Force Quasirandomness}, 
  author = {Gabriel Crudele, Peter Dukes, and Jonathan A. Noel},
  plaintextauthor = {Gabriel Crudele, Peter Dukes, Jonathan A. Noel},
  keywords = {permutation patterns, quasirandomness, independence testing, permuton, permutation limits, flag algebras},
}   

\dajEDITORdetails{%
   year={2024},
   number={8},
   received={16 March 2023},   
   published={4 October 2024},  
   doi={10.19086/da.122973},       
}   

\begin{document}

\begin{frontmatter}[classification=text]


\author[gabriel]{Gabriel Crudele\thanks{This work was completed while the first author was an undergraduate student at the University of Victoria.}}
\author[peter]{Peter Dukes\thanks{Research supported by NSERC Discovery Grant  RGPIN-2017--03891.}}
\author[jon]{Jonathan A. Noel\thanks{Research supported by NSERC Discovery Grant RGPIN-2021-02460 and NSERC Early Career Supplement DGECR-2021-00024 and a Start-Up Grant from the University of Victoria.}}

\begin{abstract}
A sequence $\pi_1,\pi_2,\dots$ of permutations is said to be \emph{quasirandom} if the induced density of every permutation $\sigma$ in $\pi_n$ converges to $1/|\sigma|!$ as $n\to\infty$. We prove that $\pi_1,\pi_2,\dots$ is quasirandom if and only if the density of each permutation $\sigma$ in the set
\[\{\perm{123},\perm{321},\perm{2143},\perm{3412},\perm{2413},\perm{3142}\}\]
converges to $1/|\sigma|!$. Previously, the smallest cardinality of a set with this property, called a \emph{quasirandom-forcing} set, was known to be between four and eight. In fact, we show that there is a single linear expression of the densities of the six permutations in this set which forces quasirandomness and show that this is best possible in the sense that there is no shorter linear expression of permutation densities with positive coefficients with this property. In the language of theoretical statistics, this expression provides a new nonparametric independence test for bivariate continuous distributions related to Spearman's $\rho$.
\end{abstract}
\end{frontmatter}

\section{Introduction}

A pervasive theme in combinatorics is that certain global characteristics of a large discrete structure are driven by a small number of local statistics. A classical result in this vein says that if $G_1,G_2,\dots$ is a sequence of graphs such that $G_n$ has $(1+o(1))|V(G_n)|^2/4$ edges and $(1+o(1))|V(G_n)|^4/16$ labelled $4$-cycles, then the global edge distribution of $G_n$ must be close to uniform; such a graph is said to be \emph{quasirandom}. More generally (and loosely) speaking, a discrete structure is said to be ``quasirandom'' if it satisfies certain properties that hold with high probability in a random structure of the same type. Quasirandom graphs were first studied in the 1980s by R\"odl~\cite{Rodl86}, Thomason~\cite{Thomason87} and Chung, Graham and Wilson~\cite{ChungGrahamWilson89}. Since then, the philosophy of quasirandom graphs has been extended to many other types of discrete structures such as hypergraphs~\cite{HavilandThomason89,ChungGraham90,KohayakawaRodlSkokan02}, tournaments~\cite{ChungGraham91tourn,CoreglianoRazborov17,CoreglianoParenteSato19,KalyanasundaramShapira13,Hancock+23,BucicLongShapiraSudakov21}, oriented graphs~\cite{Griffiths13}, groups~\cite{Gowers08} and latin squares~\cite{Cooper+22,Garbe+19}.

Our focus in this paper is on providing a new measure of quasirandomness in permutations based on local statistics. A \emph{permutation} is a bijection $\pi:[n]\to[n]$ for some $n\in\mathbb{N}$, where $[n]:=\{1,\dots,n\}$. The \emph{order} of a permutation $\pi$, denoted $|\pi|$, is the cardinality of its domain. The set of all permutations of order $n$ is denoted by $S_n$. We often depict a permutation $\pi$ of order $n$ as a word $\pi(1)\pi(2)\cdots \pi(n)$ obtained by concatenating the images of the elements of $[n]$ under $\pi$.  For example, `$\perm{123}$' denotes the identity permutation of order $3$; note that we have used a different font to distinguish these (symbolic) integers from others.  Given permutations $\sigma$ and $\pi$, of orders $k$ and $n$ respectively, a \emph{copy} of $\sigma$ in $\pi$ (as a \emph{subpermutation}) is a set $S=\{x_1,\dots,x_k\}\subseteq[n]$ such that $x_1<\cdots<x_k$ and $\pi(x_i)<\pi(x_j)$ if and only if $\sigma(i)<\sigma(j)$. A common viewpoint is to consider $\sigma$ as a \emph{permutation pattern} and examine its presence in a larger permutation $\pi$. Following~\cite{SliacanStromquist17}, let $\#(\sigma,\pi)$ be the number of copies of $\sigma$ in $\pi$. For permutations $\pi$ and $\sigma$ with $|\sigma|\leq |\pi|$, the \emph{density} of $\sigma$ in $\pi$ is defined to be
\[d(\sigma,\pi):=\frac{\#(\sigma,\pi)}{\binom{|\pi|}{|\sigma|}}.\]
In other words, $d(\sigma,\pi)$ is the probability that a randomly selected $|\sigma|$-subset of the domain of $\pi$ induces a copy of $\sigma$. For completeness, when $|\sigma|>|\pi|$, we set $d(\sigma,\pi)=0$. We extend the function $d(\cdot,\pi)$ to (formal) linear combinations of permutations linearly; i.e. $d\left(\sum_{i=1}^rc_i\cdot \sigma_i,\pi\right) = \sum_{i=1}^rc_i\cdot d(\sigma_i,\pi)$ for $c_1,\dots,c_r\in\mathbb{R}$ and permutations $\sigma_1,\dots,\sigma_r$. Following~\cite{Cooper04}, a sequence $(\pi_n)_{n=1}^{\infty}$ of permutations is \emph{quasirandom} if 
\begin{equation}\label{eq:densityLim}\lim_{n\to\infty}d(\sigma,\pi_n)=1/|\sigma|!\end{equation}
for every permutation $\sigma$. Of course, a sequence of uniformly random permutations of increasing orders is quasirandom with probability one; thus, a sequence is quasirandom if it resembles a large random permutation from the perspective of its subpermutation counts. 

Graham (see~\cite{Cooper04}) asked whether there exists a finite set $S$ of permutations such that a sequence $(\pi_n)_{n=1}^{\infty}$ with $|\pi_n|\to\infty$ is quasirandom if and only if \eqref{eq:densityLim} holds for every $\sigma\in S$; such a set $S$ is said to be \emph{quasirandom-forcing}. Kr\'a\v{l} and Pikhurko~\cite{KralPikhurko13} showed that $S_3$ is not quasirandom-forcing but $S_4$ is; this answers Graham's question in the affirmative. Given the results of~\cite{KralPikhurko13}, a natural problem is to determine the cardinality of the smallest quasirandom-forcing set. Zhang~\cite{Zhang18+} noticed that only 16 of the 24 permutations of $S_4$ are necessary in the proof of~\cite{KralPikhurko13}. The best known lower bound comes from the result of Kure\v{c}ka~\cite{Kurecka22} that every quasirandom-forcing set of permutations has cardinality at least 4.

Say that a linear combination $\sigma=\sum_{i=1}^rc_i\sigma_i$ of permutations is \emph{quasirandom-forcing} if every sequence $(\pi_n)_{n=1}^{\infty}$ satisfying $|\pi_n|\to\infty$ and $\lim_{n\to\infty}d(\sigma,\pi_n)=\sum_{i=1}^rc_i/|\sigma_i|!$ is quasirandom. A result of Bergsma and Dassios~\cite{BergsmaDassios14} implies that there is a quasirandom-forcing linear combination of 8 permutations of order 4; in particular, their result implies that there is a quasirandom-forcing set of cardinality 8. Chan, Kr\'a\v{l}, Noel, Pehova, Sharifzadeh and Volec~\cite{Chan+20} identified three more quasirandom-forcing linear combinations of 8 permutations of order 4, and one consisting of 12 permutations of order 4. Our main result provides a quasirandom-forcing linear combination of only 6 permutations of orders 3 and 4.

\begin{thm}
\label{th:main}
The linear combination $\rho^*$ of permutations defined by 
\[\rho^*:= \perm{123}+\perm{321} + \perm{2143}+\perm{3412}+ \frac{1}{2}\left(\perm{2413} + \perm{3142}\right).\]
is quasirandom-forcing. 
\end{thm}

We also show that Theorem~\ref{th:main} is best possible in the sense that there are no quasirandom-forcing linear combinations of five or fewer permutations with non-negative coefficients.

\begin{thm}
\label{th:noSmaller}
For any permutations $\sigma_1,\dots,\sigma_5$ and real numbers $c_1,\dots,c_5 \ge 0$, the linear combination $\sum_{i=1}^5c_i\cdot \sigma_i$ is not quasirandom-forcing. 
\end{thm}

In fact, Theorem~\ref{th:noSmaller} follows from a much more general result (Theorem~\ref{th:nonZeroCover}) stated and proven in Section~\ref{sec:noSmaller}, in which the coefficients $c_1,\dots,c_5$ can take negative values under certain technical conditions. 

As discussed in~\cite{EvenZoharLeng21}, the results of this paper, as well as many of those in~\cite{Cooper04,KralPikhurko13,Chan+20,Kurecka22}, can be recast in the language of statistical independence tests. Suppose that $X$ and $Y$ are $[0,1]$-valued random variables such that the cumulative distribution function $\mathbb{P}(X\leq x,Y\leq y)$ is continuous. Let $(x_1,y_1),\dots,(x_n,y_n)$ be sampled from the joint distribution of $(X,Y)$ independently of one another. By continuity of the cdf, the probability that $x_i=x_j$ or $y_i=y_j$ for $i\neq j$ is zero. By re-ordering the indices if necessary, we may assume that $x_1<x_2<\cdots<x_n$. This naturally gives rise to a permutation $\pi_n$ of order $n$ by letting
\[\pi_n(i):=|\{j: y_j\leq y_i\}|\]
for $1\leq i\leq n$. It is not hard to see that the sequence $(\pi_n)_{n=1}^{\infty}$ is quasirandom if and only if the random variables $X$ and $Y$ are independent~\cite[p.~2295]{EvenZoharLeng21}. Hoeffding's $D$~\cite{Hoeffding48} is a well-known independence test for pairs of random variables. In the context of permutations, $D$ is a linear combination of 24 permutations of order 5 which forces quasirandomness.\footnote{This means that the aforementioned question of Graham on permutation quasirandomness was technically answered, in a different setting, almost 60 years before it was asked (and about 40 years before the term ``quasirandomness'' entered the combinatorial vernacular).} Working in the language of statistical independence testing, Yanagimoto~\cite{Yanagimoto70} proved results which are very similar to those of~\cite{KralPikhurko13}. The aforementioned theorem of Bergsma and Dassios~\cite{BergsmaDassios14}, also proven in the context of independence testing, says that the following beautiful linear combination is quasirandom-forcing:
\[\tau^*:=\perm{1234}+\perm{1243}+\perm{2134}+\perm{2143}+\perm{3412}+\perm{3421}+\perm{4312}+\perm{4321}.\]
The name $\tau^*$ comes from its resemblance to the expression $\tau:=\perm{12}-\perm{21}$, known as \emph{Kendall's rank correlation coefficient} or, simply, \emph{Kendall's $\tau$}~\cite{Kendall38}. Analogously, we have chosen the name $\rho^*$ for our new independence test due to its resemblance to another very popular measure of rank correlation, known as \emph{Spearman's $\rho$}~\cite{Spearman04}, which is given by
\[\rho := \perm{123} - \perm{321} + \perm{132} - \perm{231} + \perm{213} - \perm{312}.\]

The rest of the paper is organized as follows. In Section~\ref{sec:prelim}, we build up some background on the theory of permutation limits and introduce the method of flag algebras in the context of permutation density problems. In Section~\ref{sec:flags}, we use the flag algebra method to prove Theorem~\ref{th:main}. Then, in Section~\ref{sec:noSmaller}, we turn our attention to proving Theorem~\ref{th:noSmaller}. We conclude the paper in Section~\ref{sec:concl} with some final remarks and open problems. The proofs of Theorems~\ref{th:main} and~\ref{th:noSmaller} both involve a large number of computations which are verified by computer; some of the data used in these computations, and links to our computer programs, are included in appendices that we have uploaded as ancillary files with the arxiv preprint of this paper; see \url{https://arxiv.org/src/2303.04776v4/anc/FlagAppendix.pdf} for Appendix~A and \url{https://arxiv.org/src/2303.04776v4/anc/OptimalityAppendix.pdf} for Appendices~B--E. 

\section{Permutation Limits and Flag Algebra Basics}
\label{sec:prelim}

Our results are best understood in the language of permutation limits introduced in~\cite{Hoppen+13}. A \emph{permuton} is a Borel measure $\mu$ on $[0,1]^2$ with uniform marginals, meaning that, for any Borel set $X\subseteq [0,1]$, we have $\mu(X\times[0,1])=\lambda(X)=\mu([0,1]\times X)$, where $\lambda$ is the uniform measure on $[0,1]$. Permutons will serve as ``limit objects'' for sequences of finite permutations. Permutation limit theory fits into the broader area of ``combinatorial limits'' which was developed over the past two decades and has led to important advancements in combinatorics and beyond. See the book of Lov\'asz~\cite{Lovasz12} for an in-depth treatment of graph limit theory, which is closely related to permutation limit theory. 

The \emph{support} of a Borel probability measure is the set of all points $x$ with the property that every neighbourhood of $x$ has positive measure. For a permuton $\mu$ and a permutation $\pi$ of order $n$, a \emph{copy} of $\pi$ in $\mu$ is a set of points $\{(x_1,y_1),\dots,(x_n,y_n)\}$ in the support of $\mu$ with $x_1<\cdots <x_n$ such that $y_i<y_j$ if and only if $\pi(i)<\pi(j)$. To \emph{sample} a permutation of order $n$ according to $\mu$, we mean to independently sample $n$ points according to $\mu$ and take the unique permutation of order $n$ which occurs as a copy at these points. Because $\mu$ has uniform marginals, the cumulative distribution function of $\mu$ is continuous; thus, the probability that such a set of points does not form a copy of any permutation in $\mu$ (i.e. the probability that either $x_i=x_j$ or $y_i=y_j$ for some $i\neq j$ or $(x_i,y_i)$ does not lie in the support of $\mu$) is equal to 0. The \emph{density} of $\pi$ in $\mu$, written $d(\pi,\mu)$, is the probability that a permutation of order $n$ sampled from $\mu$ forms a copy of $\pi$ in $\mu$. By the Law of Total Probability,
\begin{equation}\label{eq:sum=1}\sum_{\pi\in S_n}d(\pi,\mu)=1\end{equation}
for any permuton $\mu$ and $n\in\mathbb{N}$. As in the case of permutations, we extend $d(\,\cdot\,,\mu)$ to linear combinations of permutations linearly. That is, $d\left(\sum_{i=1}^r c_i\pi_i,\mu\right) := \sum_{i=1}^rc_i\cdot d(\pi_i,\mu)$. 

It is often useful to represent the density of a permutation as a linear combination of densities of longer permutations. For any permutation $\sigma$, permuton $\mu$ and integer $n\geq |\sigma|$, 
\begin{equation}
    \label{eq:projectUp}
d(\sigma,\mu) = \sum_{\pi\in S_n}d(\sigma,\pi)\cdot d(\pi,\mu).
\end{equation}
This follows from a simple ``two step sampling'' argument. That is, suppose that we sample $n$ points $(x_1,y_1),\dots,(x_n,y_n)$ independently according to $\mu$ and take a random set $S$ of size $|\sigma|$ from $[n]$. Then $d(\sigma,\mu)$ is the probability that $\{(x_i,y_i):i\in S\}$ forms a copy of $\sigma$. On the other hand, $d(\sigma,\pi)\cdot d(\pi,\mu)$ is equal to the probability that $\{(x_i,y_i):i\in S\}$ is a copy of $\sigma$ and $\{(x_1,y_1),\dots,(x_n,y_n)\}$ is a copy of $\pi$. Thus, by the Law of Total Probability, summing $d(\sigma,\pi)\cdot d(\pi,\mu)$ over all possible choices of $\pi$ yields $d(\sigma,\mu)$. The following definitions are fundamental to the theory of permutation limits. 

\begin{defn}
Let $(\pi_n)_{n=1}^\infty$ be a sequence of permutations such that $|\pi_n|\to\infty$. We say that $(\pi_n)_{n=1}^\infty$ is \emph{convergent} if the sequence $(d(\sigma,\pi_n))_{n=1}^\infty$ converges for every permutation $\sigma$.
\end{defn}

\begin{defn}
Let $(\pi_n)_{n=1}^\infty$ be a sequence of permutations such that $|\pi_n|\to\infty$ and let $\mu$ be a permuton. We say that $(\pi_n)_{n=1}^\infty$ \emph{converges} to $\mu$ if $\lim_{n\to\infty}d(\sigma,\pi_n)= d(\sigma,\mu)$ for every permutation $\sigma$.
\end{defn}

We will make use of the following result of~\cite{Hoppen+13}.

\begin{thm}[Hoppen et al.~\cite{Hoppen+13}]
\label{th:convergence}
For any convergent sequence $(\pi_n)_{n=1}^\infty$ of permutations, there is a unique permuton $\mu$ such that $(\pi_n)_{n=1}^\infty$ converges to $\mu$. Conversely, if $\mu$ is a permuton and $\pi_1,\pi_2,\dots$ are permutations with $|\pi_n|\to\infty$ which are sampled according to $\mu$, then $(\pi_n)_{n=1}^\infty$ converges to $\mu$ with probability $1$. 
\end{thm}

Theorem~\ref{th:convergence} yields the following corollary. 

\begin{cor}
\label{cor:qrunif}
A sequence $(\pi_n)_{n=1}^\infty$ is quasirandom if and only if it converges to the uniform measure on $[0,1]^2$. 
\end{cor}

\begin{proof}
If $\lambda$ is the uniform measure on $[0,1]^2$, then it is easy to see that $d(\sigma,\lambda)=1/|\sigma|!$ for every permutation $\sigma$. Thus, since permutation limits are unique by Theorem~\ref{th:convergence}, a sequence $(\pi_n)_{n=1}^\infty$ is quasirandom if and only if it converges to $\lambda$.
\end{proof}

In the rest of this section, we build up some of the aspects of the flag algebra method for permutations needed in this paper. We remark that this method, introduced by Razborov~\cite{Razborov07}, is adaptable to many other combinatorial objects; e.g. graphs~\cite{Razborov08}, hypergraphs~\cite{BaloghClemenLidicky22}, digraphs~\cite{HladkyKralNorin17}, tournaments~\cite{CoreglianoRazborov17}, and so on. See~\cite{Razborov13} for a survey of some of the early breakthroughs obtained via the flag algebra method.

A \emph{rooted permutation} $(\pi,R)$ of order $n$ is a permutation $\pi$ of order $n$, together with a set $R\subseteq [n]$ of distinguished points called the \emph{roots}. Given a permutation $\tau$, a \emph{$\tau$-rooted permutation} is a rooted permutation $(\pi,R)$ such that $R$ is a copy of $\tau$ in $\pi$. We write $S_n^\tau$ for the set of all $\tau$-rooted permutations of order $n$. Notationally, we express a rooted permutation $(\pi,R)$ by writing $\pi$ as a word and then underlining the positions of the roots. For example, if $\tau=\perm{12}$, then $\perm{1\underline{3}2\underline{4}}$ denotes the element $(\perm{1324},\{2,4\})$ of $S_{4}^{\tau}$. 

\begin{defn}
Let $\mathcal{A}$ and $\mathcal{A}^\tau$ be the sets of all formal linear combinations of finitely many ordinary and $\tau$-rooted permutations, respectively.
\end{defn}

Analogously, a \emph{rooted permuton} is a pair $(\mu,\mathcal{R})$ where $\mu$ is a permuton and $\mathcal{R}$ is a subset of the support of $\mu$ such that all of the $x$-coordinates of elements of $\mathcal{R}$ are pairwise distinct, as are the $y$-coordinates; the elements of $\mathcal{R}$ are called the \emph{roots}. For a permutation $\tau$, we say that $(\mu,\mathcal{R})$ is \emph{$\tau$-rooted} if $\mathcal{R}$ forms a copy of $\tau$ in $\mu$. 

If $(\pi,R)$ is a rooted permutation of order $n$ and $(\mu,\mathcal{R})$ is a rooted permuton, then a \emph{copy} of $(\pi,R)$ in $(\mu,\mathcal{R})$ is a copy $S$ of $\pi$ in $\mu$ such that $\mathcal{R}\subseteq S$ and, for each $j\in [n]$, the element of $S$ with the $j$th smallest $x$-coordinate is in $\mathcal{R}$ if and only if $j\in R$. In other words, the elements of $\mathcal{R}$ occur in the same relative positions in $S$ as the elements of $R$ do in $[n]$. To \emph{sample} a rooted permutation of order $n$ from $(\mu,\mathcal{R})$ means to sample a set $T$ of $n-|\mathcal{R}|$ points independently according to $\mu$ and take the unique rooted permutation of order $n$ induced by $T\cup \mathcal{R}$. The \emph{density}  $d((\pi,R),(\mu,\mathcal{R}))$ of $(\pi,R)$ in $(\mu,\mathcal{R})$ is the probability that sampling a rooted permutation of order $n$ from $(\mu,\mathcal{R})$ results in $(\pi,R)$; in particular, if $(\pi,R)$ is $\tau$-rooted and $(\mu,\mathcal{R})$ is $\tau'$-rooted where $\tau\neq\tau'$, then $d((\pi,R),(\mu,\mathcal{R}))=0$. As usual, we extend the density function to linear combinations of rooted permutations in a linear way. 

Next, we define $\llbracket\cdot\rrbracket:\bigcup_\tau\mathcal{A}^\tau\to \mathcal{A}$ by setting 
\[\llbracket(\pi,R)\rrbracket=\binom{n}{|R|}^{-1}\cdot\pi\] 
for every rooted permutation $(\pi,R)$ of order $n$ and extending $\llbracket\cdot\rrbracket$ to all of $\bigcup_\tau\mathcal{A}^\tau$ linearly. This is a standard type of ``unrooting'' operation that appears in applications of flag algebras involving other types of combinatorial structures as well~\cite{Razborov07}; it is particularly simple in the case of permutations. Now, we claim that, if $\mathcal{R}$ is a set of cardinality $|\tau|$ sampled according to $\pi$, then
\begin{equation}\label{eq:unroot}\mathbb{E}_\mathcal{R}(d((\pi,R),(\mu,\mathcal{R})))= d\left(\llbracket(\pi,R)\rrbracket,\mu\right).\end{equation}
Indeed, the right side is equal to $\binom{n}{|\tau|}^{-1}\cdot d(\pi,\mu)$ while the left side is the probability that, upon sampling $n$ points according to $\mu$, one obtains a copy of $\pi$ in which the first $n-|\tau|$ points chosen correspond to the elements of $[n]\setminus R$ and the last $|\tau|$ points correspond to the elements of $R$. This is also equal to $\binom{n}{|\tau|}^{-1}\cdot d(\pi,\mu)$.

Let us now describe a multiplication operation $\times$ on $\mathcal{A}^\tau$ which is respected by $d(\cdot,(\mu,\mathcal{R}))$ for any $\tau$-rooted permuton $(\mu,\mathcal{R})$. That is, we want that
\begin{equation}\label{eq:preserveProds}
d(A\times B,(\mu,\mathcal{R})) = d(A,(\mu,\mathcal{R}))\cdot d(B,(\mu,\mathcal{R}))
\end{equation}
for all $A,B\in\mathcal{A}^\tau$. Although Razborov~\cite{Razborov07} defines a more general multiplication which works for many different classes of combinatorial objects, here, we only describe the specific instance of it needed for multiplying linear combinations of rooted permutations. First, we declare that our product distributes over addition. Hence, it is sufficient to define the product of two $\tau$-rooted permutations $(\pi_1,{R_1})$ and $(\pi_2,{R_2})$. 

In order to ensure our product is respected by $d(\,\cdot\,,(\mu,\mathcal{R}))$ for any $\tau$-rooted permuton $(\mu,\mathcal{R})$, we would like to represent $(\pi_1,{R_1})\times(\pi_2,{R_2})$ as a linear combination of $\tau$-rooted permutations in such a way that the equality \[d((\pi_1,{R_1})\times(\pi_2,{R_2}),(\mu,\mathcal{R})) = d((\pi_1,{R_1}),(\mu,\mathcal{R}))\cdot d((\pi_2,{R_2}),(\mu,\mathcal{R}))\] 
holds automatically, regardless of the underlying $\tau$-rooted permuton $(\mu,\mathcal{R})$. To this end, we think of the product $d((\pi_1,{R_1}),(\mu,\mathcal{R}))\cdot d((\pi_2,{R_2}),(\mu,\mathcal{R}))$ as the probability that, upon sampling a set $T_1$ of $(|\pi_1|-|\tau|)$ points and a set  $T_2$ of $(|\pi_2|-|\tau|)$ points according to $\mu$ independently, $T_1\cup\mathcal{R}$ forms a copy of $(\pi_1,{R_1})$ and $T_2\cup\mathcal{R}$ forms a copy of $(\pi_2,{R_2})$; let $E_1$ and $E_2$ be the events that the first and second statements hold, respectively. Given that $E_1\cap E_2$ holds, $T_1\cup T_2\cup\mathcal{R}$ forms a copy of some $\tau$-rooted permutation $(\pi_3,R_3)$ of order $|\pi_1|+|\pi_2|-|\tau|$. Note that different samples may yield different rooted permutations $(\pi_3,R_3)$. For example, if $(\pi_1,{R_1})=(\pi_2,{R_2})=\perm{\underline{1}2}$, then $(\pi_3,R_3)$ may be equal to $\perm{\underline{1}23}$ or $\perm{\underline{1}32}$. Letting $m=|\pi_1|+|\pi_2|-|\tau|$ and conditioning on the possible outcomes for $(\pi_3,R_3)$, we get that 
\[d((\pi_1,{R_1}),(\mu,\mathcal{R}))\cdot d((\pi_2,{R_2}),(\mu,\mathcal{R})) = \mathbb{P}(E_1\cap E_2)\]
\[=\sum_{\substack{(\sigma,Q)\in S_m^\tau\\d((\sigma,Q),(\mu,\mathcal{R}))>0}} \mathbb{P}(E_1\cap E_2\cap\{(\pi_3,R_3)=(\sigma,Q)\})\]
\[= \sum_{\substack{(\sigma,Q)\in S_m^\tau\\d((\sigma,Q),(\mu,\mathcal{R}))>0}}\mathbb{P}(E_1\cap E_2\mid (\pi_3,R_3)=(\sigma,Q))\cdot d((\sigma,Q),(\mu,\mathcal{R})).
\]
The key observation is that, for any $(\sigma,Q)\in S_m^\tau$, the expression $\mathbb{P}(E_1\cap E_2\mid (\pi_3,R_3)=(\sigma,Q))$ depends only on $(\pi_1,{R_1}), (\pi_2,{R_2})$ and $(\sigma,Q)$, not on the underlying permuton $(\mu,\mathcal{R})$. Indeed, $\mathbb{P}(E_1\cap E_2\mid (\pi_3,R_3)=(\sigma,Q))$ is nothing more than the probability that, in a random partition of $[m]\setminus Q$ into sets of cardinalities $|\pi_1|-|\tau|$ and $|\pi_2|-|\tau|$, the union of the first set with $Q$ forms a copy of $(\pi_1,{R_1})$ in $(\sigma,Q)$, and the analogous statement holds for the second set. For each $(\sigma,Q)\in S_m^\tau$, we let $c_{(\sigma,Q)}$ denote this quantity and define
\[(\pi_1,{R_1})\times (\pi_2,{R_2})=\sum_{(\sigma,Q)\in S_m^\tau}c_{(\sigma,Q)}\cdot (\sigma,Q).\]
Then, by construction, this multiplication operation satisfies $d(A\times B,(\mu,\mathcal{R}))=d(A,(\mu,\mathcal{R}))\cdot d(B,(\mu,\mathcal{R}))$ for any $A,B\in \mathcal{A}^\tau$ and $\tau$-rooted permuton $(\mu,\mathcal{R})$. We now present an observation which is crucial to our proof of Theorem~\ref{th:main}. 

\begin{obs}
\label{obs:xMx}
Let $\tau$ be a permutation and $\mu$ be a permuton such that $d(\tau,\mu)>0$ and let $k\geq1$. Given a vector $x:=(x_1,\dots,x_k)\in (\mathcal{A}^\tau)^k$ and a positive semi-definite matrix $M$ with real entries, we have
\[d(\llbracket x M x^T\rrbracket,\mu) \geq 0.\]
Moreover, if $M$ is positive definite, then equality holds if and only if $d\left(x_i,(\mu,\mathcal{R})\right)=0$ for all $1\leq i\leq k$ for every copy $\mathcal{R}$ of $\tau$ in $\mu$.
\end{obs}

\begin{proof}
By linearity of $\llbracket\cdot\rrbracket$ and $d(\cdot,\mu)$, we have
\[d(\llbracket x M x^T\rrbracket,\mu)= \sum_{i=1}^k\sum_{j=1}^kM_{i,j}\cdot d\left(\left\llbracket x_i\times x_j\right\rrbracket,\mu\right).\]
By \eqref{eq:unroot}, linearity of expectation and \eqref{eq:preserveProds}, this is equal to 
\[\sum_{i=1}^k\sum_{j=1}^kM_{i,j}\cdot \mathbb{E}_\mathcal{R}\left(d\left(x_i\times x_j,(\mu,\mathcal{R})\right)\right)\]
\[=\mathbb{E}_\mathcal{R}\left(\sum_{i=1}^k\sum_{j=1}^kM_{i,j}\cdot d\left(x_i,(\mu,\mathcal{R})\right)\cdot d\left(x_j,(\mu,\mathcal{R})\right)\right)\]
which is non-negative because $M$ is positive semidefinite. If we additionally assume that $M$ is positive definite, then the expression inside of the expectation is positive unless $d\left(x_i,(\mu,\mathcal{R})\right)=0$ for all $1\leq i\leq k$. Thus, in order for equality to hold, we must have $d\left(x_i,(\mu,\mathcal{R})\right)=0$ for all $1\leq i\leq k$ for every copy $\mathcal{R}$ of $\tau$ in $\mu$.
\end{proof}

The aspects of the flag algebra method described in this section are sufficient for presenting the proof of Theorem~\ref{th:main}. However, the true power of the method comes from using these ideas to devise a semidefinite program, that can be run on a computer, to search for highly complicated proofs of extremal inequalities for permutation densities. We discuss this aspect of the method at the end of the next section.

\section{Proof of Theorem~\ref{th:main}}
\label{sec:flags}

Our goal in this section is to prove Theorem~\ref{th:main}. We will mainly focus on proving the following analytic analogue and derive Theorem~\ref{th:main} from it near the end of the section.

\begin{thm}\label{th:analyticMain}
For every permuton $\mu$, it holds that $d(\rho^*,\mu)\geq 11/24$. Moreover, $d(\rho^*,\mu)= 11/24$ if and only if $\mu$ is the uniform permuton. 
\end{thm}

To characterize the case of equality, we make use of the following result of~\cite{Chan+20}.

\begin{lem}[Chan et al.~{\cite[Lemmas 3, 4 and 5]{Chan+20}}]
\label{lem:z1z2}
Let
\[z_1:=(\perm{\underline{1}2\underline{3}4} - \perm{\underline{1}4\underline{3}2})
+(\perm{1\underline{2}3\underline{4}} - \perm{3\underline{2}1\underline{4}})
+(\perm{\underline{2}3\underline{4}1} - \perm{\underline{2}1\underline{4}3})
+(\perm{4\underline{1}2\underline{3}} - \perm{2\underline{1}4\underline{3}}),\]
\[z_2:=(\perm{4\underline{3}2\underline{1}} - \perm{2\underline{3}4\underline{1}})
+(\perm{\underline{4}3\underline{2}1} - \perm{\underline{4}1\underline{2}3})
+(\perm{\underline{3}2\underline{1}4} - \perm{\underline{3}4\underline{1}2})
+(\perm{1\underline{4}3\underline{2}} - \perm{3\underline{4}1\underline{2}}).\]
If $\mu$ is a permuton such that $\mu(z_1,(\mu,\mathcal{R}_1))=0$ for every copy $\mathcal{R}_1$ of $\perm{12}$ in $\mu$ and $\mu(z_2,(\mu,\mathcal{R}_2))=0$ for every copy $\mathcal{R}_2$ of $\perm{21}$ in $\mu$, then $\mu$ is the uniform measure. 
\end{lem}

We are now prepared to prove Theorem~\ref{th:analyticMain}.

\begin{proof}[Proof of Theorem~\ref{th:analyticMain}]
Fix a permuton $\mu$. Let $\tau_1=\perm{12}$ and $\tau_2=\perm{21}$ and let $x=(x_1,\dots,x_5)\in(\mathcal{A}^{\tau_1})^5$ be the vector whose entries are defined as follows:
\begin{align*}
x_1&:=(\perm{\underline{1}\underline{2}34} - \perm{\underline{1}\underline{2}43})
+(\perm{\underline{1}23\underline{4}} - \perm{\underline{1}32\underline{4}})
+(\perm{1\underline{2}\underline{3}4} - \perm{4\underline{2}\underline{3}1})
+(\perm{12\underline{3}\underline{4}} - \perm{21\underline{3}\underline{4}})
+(\perm{\underline{1}34\underline{2}} - \perm{\underline{1}43\underline{2}})\\
&+(\perm{\underline{1}\underline{4}23} - \perm{\underline{1}\underline{4}32})
+(\perm{23\underline{1}\underline{4}} - \perm{32\underline{1}\underline{4}})
+(\perm{\underline{2}\underline{3}41} - \perm{\underline{2}\underline{3}14})
+(\perm{2\underline{3}\underline{4}1} - \perm{1\underline{3}\underline{4}2})
+(\perm{\underline{2}41\underline{3}} - \perm{\underline{2}14\underline{3}})\\
&+(\perm{\underline{3}12\underline{4}} - \perm{\underline{3}21\underline{4}})
+(\perm{3\underline{1}\underline{4}2} - \perm{2\underline{1}\underline{4}3})
+(\perm{\underline{3}\underline{4}12} - \perm{\underline{3}\underline{4}21})
+(\perm{34\underline{1}\underline{2}} - \perm{43\underline{1}\underline{2}})
+(\perm{4\underline{1}\underline{2}3} - \perm{3\underline{1}\underline{2}4})\\
&+(\perm{41\underline{2}\underline{3}} - \perm{14\underline{2}\underline{3}}),\\
x_2&:=(\perm{\underline{1}2\underline{3}4} - \perm{\underline{1}4\underline{3}2})
+(\perm{1\underline{2}3\underline{4}} - \perm{3\underline{2}1\underline{4}})
+(\perm{\underline{2}3\underline{4}1} - \perm{\underline{2}1\underline{4}3})
+(\perm{4\underline{1}2\underline{3}} - \perm{2\underline{1}4\underline{3}}),\\
x_3&:=(\perm{\underline{1}24\underline{3}} - \perm{\underline{1}42\underline{3}})
+(\perm{1\underline{2}\underline{4}3} - \perm{3\underline{2}\underline{4}1})
+(\perm{\underline{1}\underline{3}24} - \perm{\underline{1}\underline{3}42})
+(\perm{1\underline{3}2\underline{4}} - \perm{2\underline{3}1\underline{4}})
+(\perm{1\underline{3}\underline{4}2} - \perm{2\underline{3}\underline{4}1})\\
&+(\perm{14\underline{2}\underline{3}} - \perm{41\underline{2}\underline{3}})
+(\perm{\underline{2}1\underline{3}4} - \perm{\underline{2}4\underline{3}1})
+(\perm{2\underline{1}3\underline{4}} - \perm{3\underline{1}2\underline{4}})
+(\perm{\underline{2}1\underline{4}3} - \perm{\underline{2}3\underline{4}1})
+(\perm{\underline{2}14\underline{3}} - \perm{\underline{2}41\underline{3}})\\
&+(\perm{2\underline{1}\underline{4}3} - \perm{3\underline{1}\underline{4}2})
+(\perm{2\underline{1}4\underline{3}} - \perm{4\underline{1}2\underline{3}})
+(\perm{\underline{2}\underline{3}14} - \perm{\underline{2}\underline{3}41})
+(\perm{24\underline{1}\underline{3}} - \perm{42\underline{1}\underline{3}})
+(\perm{3\underline{1}\underline{2}4} - \perm{4\underline{1}\underline{2}3})\\
&+(\perm{3\underline{1}4\underline{2}} - \perm{4\underline{1}3\underline{2}}),\\
x_4&:=(\perm{1\underline{3}2\underline{4}} - \perm{2\underline{3}1\underline{4}})
+(\perm{13\underline{2}\underline{4}} - \perm{31\underline{2}\underline{4}})
+(\perm{1\underline{3}\underline{4}2} - \perm{2\underline{3}\underline{4}1})
+(\perm{14\underline{2}\underline{3}} - \perm{41\underline{2}\underline{3}})
+(\perm{\underline{2}1\underline{3}4} - \perm{\underline{2}4\underline{3}1})\\
&+(\perm{\underline{2}13\underline{4}} - \perm{\underline{2}31\underline{4}})
+(\perm{2\underline{1}\underline{3}4} - \perm{4\underline{1}\underline{3}2})
+(\perm{2\underline{1}3\underline{4}} - \perm{3\underline{1}2\underline{4}})
+(\perm{\underline{2}1\underline{4}3} - \perm{\underline{2}3\underline{4}1})
+(\perm{\underline{2}14\underline{3}} - \perm{\underline{2}41\underline{3}})\\
&+(\perm{2\underline{1}\underline{4}3} - \perm{3\underline{1}\underline{4}2})
+(\perm{2\underline{1}4\underline{3}} - \perm{4\underline{1}2\underline{3}})
+(\perm{\underline{2}\underline{3}14} - \perm{\underline{2}\underline{3}41})
+(\perm{\underline{2}\underline{4}13} - \perm{\underline{2}\underline{4}31})
+(\perm{3\underline{1}\underline{2}4} - \perm{4\underline{1}\underline{2}3})\\
&+(\perm{3\underline{1}4\underline{2}} - \perm{4\underline{1}3\underline{2}}),\\
x_5&:=(\perm{1\underline{3}2\underline{4}} - \perm{2\underline{3}1\underline{4}})
+(\perm{\underline{1}3\underline{4}2} - \perm{\underline{1}2\underline{4}3})
+(\perm{\underline{1}4\underline{2}3} - \perm{\underline{1}3\underline{2}4})
+(\perm{\underline{2}1\underline{3}4} - \perm{\underline{2}4\underline{3}1})
+(\perm{2\underline{1}3\underline{4}} - \perm{3\underline{1}2\underline{4}})\\
&+(\perm{3\underline{1}4\underline{2}} - \perm{4\underline{1}3\underline{2}})
+(\perm{\underline{3}2\underline{4}1} - \perm{\underline{3}1\underline{4}2})
+(\perm{4\underline{2}1\underline{3}} - \perm{1\underline{2}4\underline{3}}).
\end{align*}
Next, we let $y=(y_1,\dots,y_5)\in (\mathcal{A}^{\tau_2})^5$ be defined as follows. 
\begin{align*}
y_1&:=(\perm{43\underline{2}\underline{1}} - \perm{34\underline{2}\underline{1}})
+(\perm{\underline{4}32\underline{1}} - \perm{\underline{4}23\underline{1}})
+(\perm{4\underline{3}\underline{2}1} - \perm{1\underline{3}\underline{2}4})
+(\perm{\underline{4}\underline{3}21} - \perm{\underline{4}\underline{3}12})
+(\perm{32\underline{4}\underline{1}} - \perm{23\underline{4}\underline{1}})\\
&+(\perm{\underline{2}43\underline{1}} - \perm{\underline{2}34\underline{1}})
+(\perm{\underline{4}21\underline{3}} - \perm{\underline{4}12\underline{3}})
+(\perm{3\underline{2}\underline{1}4} - \perm{4\underline{2}\underline{1}3})
+(\perm{\underline{3}\underline{2}14} - \perm{\underline{3}\underline{2}41})
+(\perm{2\underline{4}\underline{1}3} - \perm{3\underline{4}\underline{1}2})\\
&+(\perm{\underline{4}\underline{1}32} - \perm{\underline{4}\underline{1}23})
+(\perm{\underline{3}14\underline{2}} - \perm{\underline{3}41\underline{2}})
+(\perm{\underline{2}\underline{1}43} - \perm{\underline{2}\underline{1}34})
+(\perm{21\underline{4}\underline{3}} - \perm{12\underline{4}\underline{3}})
+(\perm{14\underline{3}\underline{2}} - \perm{41\underline{3}\underline{2}})\\
&+(\perm{1\underline{4}\underline{3}2} - \perm{2\underline{4}\underline{3}1}),\\
y_2&:=(\perm{4\underline{3}2\underline{1}} - \perm{2\underline{3}4\underline{1}})
+(\perm{\underline{4}3\underline{2}1} - \perm{\underline{4}1\underline{2}3})
+(\perm{\underline{3}2\underline{1}4} - \perm{\underline{3}4\underline{1}2})
+(\perm{1\underline{4}3\underline{2}} - \perm{3\underline{4}1\underline{2}}),\\
y_3&:=(\perm{3\underline{4}2\underline{1}} - \perm{2\underline{4}3\underline{1}})
+(\perm{\underline{3}4\underline{2}1} - \perm{\underline{3}1\underline{2}4})
+(\perm{4\underline{2}3\underline{1}} - \perm{3\underline{2}4\underline{1}})
+(\perm{\underline{4}\underline{2}31} - \perm{\underline{4}\underline{2}13})
+(\perm{\underline{3}\underline{2}41} - \perm{\underline{3}\underline{2}14})\\
&+(\perm{2\underline{4}\underline{3}1} - \perm{1\underline{4}\underline{3}2})
+(\perm{4\underline{3}\underline{1}2} - \perm{2\underline{3}\underline{1}4})
+(\perm{\underline{4}31\underline{2}} - \perm{\underline{4}13\underline{2}})
+(\perm{\underline{3}4\underline{1}2} - \perm{\underline{3}2\underline{1}4})
+(\perm{3\underline{4}\underline{1}2} - \perm{2\underline{4}\underline{1}3})\\
&+(\perm{\underline{3}41\underline{2}} - \perm{\underline{3}14\underline{2}})
+(\perm{3\underline{4}1\underline{2}} - \perm{1\underline{4}3\underline{2}})
+(\perm{4\underline{2}\underline{1}3} - \perm{3\underline{2}\underline{1}4})
+(\perm{2\underline{4}1\underline{3}} - \perm{1\underline{4}2\underline{3}})
+(\perm{41\underline{3}\underline{2}} - \perm{14\underline{3}\underline{2}})\\
&+(\perm{31\underline{4}\underline{2}} - \perm{13\underline{4}\underline{2}}),\\
y_4&:=(\perm{\underline{4}\underline{2}31} - \perm{\underline{4}\underline{2}13})
+(\perm{\underline{4}2\underline{3}1} - \perm{\underline{4}1\underline{3}2})
+(\perm{\underline{3}\underline{2}41} - \perm{\underline{3}\underline{2}14})
+(\perm{2\underline{4}\underline{3}1} - \perm{1\underline{4}\underline{3}2})
+(\perm{4\underline{3}\underline{1}2} - \perm{2\underline{3}\underline{1}4})\\
&+(\perm{\underline{4}3\underline{1}2} - \perm{\underline{4}2\underline{1}3})
+(\perm{4\underline{3}1\underline{2}} - \perm{1\underline{3}4\underline{2}})
+(\perm{\underline{4}31\underline{2}} - \perm{\underline{4}13\underline{2}})
+(\perm{\underline{3}4\underline{1}2} - \perm{\underline{3}2\underline{1}4})
+(\perm{3\underline{4}\underline{1}2} - \perm{2\underline{4}\underline{1}3})\\
&+(\perm{\underline{3}41\underline{2}} - \perm{\underline{3}14\underline{2}})
+(\perm{3\underline{4}1\underline{2}} - \perm{1\underline{4}3\underline{2}})
+(\perm{4\underline{2}\underline{1}3} - \perm{3\underline{2}\underline{1}4})
+(\perm{\underline{2}4\underline{1}3} - \perm{\underline{2}3\underline{1}4})
+(\perm{41\underline{3}\underline{2}} - \perm{14\underline{3}\underline{2}})\\
&+(\perm{31\underline{4}\underline{2}} - \perm{13\underline{4}\underline{2}}),\\
y_5&:=(\perm{\underline{4}\underline{2}31} - \perm{\underline{4}\underline{2}13})
+(\perm{\underline{3}24\underline{1}} - \perm{\underline{3}42\underline{1}})
+(\perm{24\underline{3}\underline{1}} - \perm{42\underline{3}\underline{1}})
+(\perm{4\underline{3}\underline{1}2} - \perm{2\underline{3}\underline{1}4})
+(\perm{\underline{4}31\underline{2}} - \perm{\underline{4}13\underline{2}})\\
&+(\perm{31\underline{4}\underline{2}} - \perm{13\underline{4}\underline{2}})
+(\perm{\underline{3}\underline{1}24} - \perm{\underline{3}\underline{1}42})
+(\perm{1\underline{4}\underline{2}3} - \perm{3\underline{4}\underline{2}1}).
\end{align*}
Note that, for each $1\leq i\leq 5$, the expression $y_i$ can be obtained from $x_i$ in the following way. For each rooted permutation appearing as a term of $x_i$, take its corresponding permutation matrix and rotate it by a quarter turn (keeping track of the locations of the roots). This transformation changes each $\tau_1$-rooted permutation in the expression for $x_i$ into the corresponding $\tau_2$-rooted permutation in expression for $y_i$. Now, consider the following $5\times5$ matrix:
\[M = \frac{1}{112}
\begin{bmatrix}
86 & 6 &  40 & 40 & -40\\
6 &  136 &  46 & 46 & -46\\
40 &  46 & 101 & -17 & -38\\
40 & 46 &  -17 & 101 & -46\\
-40 &  -46 &  -38 & -46 & 101
\end{bmatrix}.\]
The matrix $M$ is positive definite; the eigenvalues of $112\cdot M$ are approximately 243.3, 118.4, 104.4, 48.1 and 10.7. Applying \eqref{eq:projectUp} in the case $n=6$ and combining it with the linearity of $d(\cdot,\mu)$, we obtain the following for every permuton $\mu$:
\begin{equation}\label{eq:twoStepSampling}
d(\rho^*,\mu) = \sum_{\pi\in S_6}d(\rho^*,\pi)\cdot d(\pi,\mu).
\end{equation}
We claim that the following equality in the vector space $\mathcal{A}$ is valid:
\begin{equation}\label{eq:rhoStarExpression}
    \sum_{\pi\in S_6}d(\rho^*,\pi)\cdot \pi  - \left\llbracket xMx^T\right\rrbracket-\left\llbracket yMy^T\right\rrbracket = \sum_{\pi\in S_6}\frac{11}{24}\cdot \pi.
\end{equation}
To verify this, one must show that the coefficient of each of the 720 elements $\pi$ of $S_6$ in the expression on the left side of \eqref{eq:rhoStarExpression} is $11/24$; this is rather tedious to do by hand. To illustrate how the computation works in principle, let us consider the example $\pi=\perm{123456}$. One can easily compute $d(\rho^*,\perm{123456})=1$. Now, the coefficient of $\perm{123456}$ in $\llbracket x_i\times x_j\rrbracket$ is equal to $7/15$ if $(i,j)=(1,1)$, is equal to $2/15$ if $(i,j)=(2,2)$, is equal to $1/5$ if $(i,j)$ is either $(1,2)$ or $(2,1)$ and is $0$ otherwise. Also, the coefficient of $\perm{123456}$ in $\llbracket y_i\times y_j\rrbracket$ is zero for all $i$ and $j$. Let $M(i,j)$ be the entry on the $i$th row and $j$th column of $M$. Putting this together, the coefficient of $\perm{123456}$ on the left side of \eqref{eq:rhoStarExpression} is
\[1 - M(1,1)\cdot\left(\frac{7}{15}\right)-M(2,2)\cdot\left(\frac{2}{15}\right)-M(1,2)\cdot\left(\frac{1}{5}\right)-M(2,1)\cdot\left(\frac{1}{5}\right)\]
\[=1 - \left(\frac{86}{112}\right)\cdot\left(\frac{7}{15}\right)-\left(\frac{136}{112}\right)\cdot\left(\frac{2}{15}\right)-\left(\frac{6}{112}\right)\cdot\left(\frac{1}{5}\right)-\left(\frac{6}{112}\right)\cdot\left(\frac{1}{5}\right)=\frac{11}{24}.\]
In case the reader wishes to repeat this calculation for a different permutation $\pi\in S_6$, we have included all 720 terms of $\sum_{\pi\in S_6}d(\rho^*,\pi)\cdot \pi$ and all of the expressions $\llbracket x_i\times x_j\rrbracket$ and $\llbracket y_i\times y_j\rrbracket$ for $1\leq i\leq j\leq 5$ in Appendix~A. We have also written a Java program which formally verifies that the coefficient of each $\pi\in S_6$ in the expression on the left side of \eqref{eq:rhoStarExpression} is equal to $11/24$; a link to a git repository containing this program can be found in Appendix~A.

By Observation~\ref{obs:xMx}, both of $d\left(\left\llbracket x^TMx\right\rrbracket,\mu\right)$ and $d\left(\left\llbracket y^TMy\right\rrbracket,\mu\right)$ are nonnegative for any permuton $\mu$. Thus, by \eqref{eq:twoStepSampling}, \eqref{eq:rhoStarExpression} and \eqref{eq:sum=1} and linearity of $d(\cdot,\mu)$,
\[
d(\rho^*,\mu) = \sum_{\pi\in S_6}d(\rho^*,\pi)\cdot d(\pi,\mu) \geq \frac{11}{24}\sum_{\pi\in S_6}d(\pi,\mu) = \frac{11}{24}.
\]
Therefore, $d(\rho^*,\mu)\geq 11/24$ for every permuton $\mu$. 

All that is left to do is to characterize the case of equality. First, if $\lambda$ is the uniform measure on $[0,1]^2$, then
\begin{align*}
d(\rho^*,\lambda) &= d(\perm{123},\lambda)+d(\perm{321},\lambda)+d(\perm{2143},\lambda)+d(\perm{3421},\lambda)+\frac{1}{2}\left( d(\perm{2413},\lambda)+d(\perm{3142},\lambda)\right)\\
&= \frac{1}{3!}+\frac{1}{3!}+\frac{1}{4!}+\frac{1}{4!}+\frac{1}{2}\left(\frac{1}{4!}+\frac{1}{4!}\right)
= \frac{11}{24}.
\end{align*}
Now, suppose that $\mu$ is a permuton such that $d(\rho^*,\mu)=11/24$. Then from \eqref{eq:rhoStarExpression}, we get
\[
d\left( \left\llbracket xMx^T\right\rrbracket,\mu \right)=d\left( \left\llbracket yMy^T\right\rrbracket,\mu \right) = 0.
\]
Since $M$ is positive definite, Observation~\ref{obs:xMx} tells us that $d(x_2,{(\mu,\mathcal{R}_1)})=0$ for every copy $\mathcal{R}_1$ of $\perm{12}$ in $\mu$ and $d(y_2,{(\mu,\mathcal{R}_2)})=0$ for every copy $\mathcal{R}_2$ of $\perm{21}$ in $\mu$. Since $x_2$ and $y_2$ are precisely equal to the expressions $z_1$ and $z_2$ in Lemma~\ref{lem:z1z2}, it follows that $\mu$ is the uniform permuton. 
\end{proof}

\begin{proof}[Proof of Theorem~\ref{th:main}]
Suppose, for the sake of contradiction, that $\rho^*$ is not quasirandom-forcing. Then there is a permutation sequence $(\pi_n)_{n=1}^\infty$ with $|\pi_n|\to\infty$, a permutation $\sigma$ and $\varepsilon>0$  such that $\lim_{n\to\infty}d(\rho^*,\pi_n)=11/24$, but $\left|d(\sigma,\pi_n)-1/|\sigma|!\right|\geq\varepsilon$ for all $n\geq1$. Since there are countably many finite permutations and $[0,1]^\mathbb{N}$ is compact, we can let $(\pi_{n_k})_{k=1}^\infty$ be a convergent subsequence of $(\pi_n)_{n=1}^\infty$. By Theorem~\ref{th:convergence}, there is a permuton $\mu$ such that $(\pi_{n_k})_{k=1}^\infty$ converges to $\mu$. In particular, $d(\rho^*,\mu) = \lim_{k\to\infty}d(\rho^*,\pi_{n_k})=11/24$ and hence, by Theorem~\ref{th:analyticMain}, $\mu$ must be the uniform permuton. However, if $\mu$ is the uniform permuton, then $d(\sigma,\mu)=1/|\sigma|!$ which contradicts the fact that $\left|d(\sigma,\pi_n)-1/|\sigma|!\right|\geq\varepsilon$ for all $n\geq1$.
\end{proof}

The proof of Theorem~\ref{th:analyticMain} given above was discovered with the aid of an SDP solver; this is a core idea of most applications of the flag algebra method. As we discussed during the proof, just to verify every detail of the proof by hand would already be unreasonably laborious; thus, it seems to us that it would be close to impossible to discover the proof by hand. This illustrates the computational advantages of the flag algebra method. Let us roughly explain how the proof was found and how similar ideas can be applied to other permutation density problems. See~\cite{Chan+20,SliacanStromquist17,Balogh+15} for other successful applications of this method to such problems.

Let $\rho$ be a finite linear combination of permutations. Now, fix $k$ permutations $\tau_1,\dots,\tau_k$ and integers $n_1,\dots,n_k$, of your own choosing, such that $n_s\geq |\tau_s|$ for all $1\leq s\leq k$. For $1\leq s\leq k$, let $N_s:=|S_{n_s}^{\tau_s}|$ be the number of $\tau_s$-rooted permutations of order $n_s$ and let $x_s$ be a row vector in $(\mathcal{A}^{\tau_s})^{N_s}$ containing every element of $S_{n_s}^{\tau_s}$ as an entry. Now, for each $1\leq s\leq k$, let $M_s$ be an $N_s\times N_s$ positive semidefinite matrix. Then, for any $n$ which is at least the maximum order of any permutation appearing in $\rho$ such that $n\geq 2n_s-|\tau_s|$ for all $1\leq s\leq k$, one can use \eqref{eq:projectUp} to write
\[\sum_{\pi\in S_n}d(\rho,\pi)\cdot \pi - \sum_{s=1}^k\llbracket x_s M_sx_s^T\rrbracket = \sum_{\pi\in S_n}c_\pi\cdot \pi\]
for some coefficients $c_\pi=c_\pi(\rho,M_1,\dots,M_k)$ for $\pi\in S_n$ which depend on the densities $d(\rho,\pi)$ and the entries in the matrices $M_1,\dots,M_k$. So, by Observation~\ref{obs:xMx} and \eqref{eq:projectUp}, for any permuton $\mu$, one gets
\[d(\rho,\mu) = \sum_{\pi\in S_n}d(\rho,\pi)\cdot d(\pi,\mu)\geq \sum_{\pi\in S_n}c_\pi\cdot d(\pi,\mu)\]
which, by \eqref{eq:sum=1}, is at least
\[\min\{c_\pi: \pi\in S_n\}.\]
Thus, for every permuton $\mu$, the density $d(\rho,\mu)$ is bounded below by the maximum over all $t\in\mathbb{R}$ satisfying $t\leq c_\pi$ for all $\pi\in S_n$ where, as we mentioned earlier, $c_\pi$ can be expressed in terms of the constants $d(\rho,\pi)$ and the entries of the positive semidefinite matrices $M_1,\dots,M_k$. The problem of maximizing $t$ subject to the constraints 
\begin{enumerate}
\stepcounter{equation}
    \item $t\leq c_\pi$ and
\stepcounter{equation}
    \item all of the matrices $M_1,\dots,M_k$ are positive semidefinite
\end{enumerate}
can be naturally phrased as a semidefinite program and attacked computationally using an SDP solver. The matrices provided as part of the output to the solver tend to contain floating point entries which need to be ``rounded'' to exact values in order to make the proof rigorous. In the case of Theorem~\ref{th:analyticMain}, it was also useful to modify the certificate of the proof to be positive definite as opposed to semidefinite; this was needed to prove that the uniform permuton is the unique minimizer of $d(\rho^*,\mu)$. This roughly describes the way in which the proof of Theorem~\ref{th:analyticMain} was found.

\section{Proof of Theorem~\ref{th:noSmaller}}
\label{sec:noSmaller}

Our goal in this section is to prove that no positive linear combination with fewer than six terms is quasirandom-forcing. In fact, we will prove a much more general result which allows a mixture of positive and negative coefficients, subject to a certain technical condition; see Theorem~\ref{th:nonZeroCover} below. 

\subsection{Cover matrices and fuzzy permutation matrices}

Given a matrix $A$, let $A(i,j)$ be the entry on the $i$th row and $j$th column of $A$. Recall that for $\sigma \in S_k$, the \emph{permutation matrix} $A_\sigma$ of $\sigma$ is the $k \times k$ matrix such that $A_\sigma(i,j)=1$ if $j = \sigma(i)$ and $A_\sigma(i,j)=0$ otherwise. For $\sigma \in S_k$ and $n \ge k$, we may use the density function $d(\sigma,\cdot)$ to induce the \emph{$n \times n$ cover matrix} of $\sigma$, defined by
\[A_\sigma^{\uparrow n}:=\sum_{\pi\in S_n}d(\sigma,\pi)\cdot A_\pi.\]
Note that $A_\sigma^{\uparrow k}$ is nothing more than $A_\sigma$. Following~\cite{Chan+20,Kurecka22}, for a linear combination $\rho=\sum_{i=1}^r c_i\sigma_i$ and $n\geq\max\{|\sigma_i|: i=1,\dots,r\}$, the \emph{$n\times n$ cover matrix} of $\rho$ is defined to be 
\[A_\rho^{\uparrow n}:=\sum_{i=1}^r c_i\cdot A_{\sigma_i}^{\uparrow n}.\]
If all of the permutations $\sigma_1,\dots,\sigma_r$ are of order $k$, then we set $A_\rho:=A_\rho^{\uparrow k}$.

Our main result in this section will be stated in terms of a close relative of cover matrices, which 
we define
using the following polynomials. 

\begin{defn}
For positive integers $j\leq k\leq n$ and $1\leq x\leq n$, define
\[f^{\uparrow n}_{k,j}(x) := \binom{x-1}{j-1} \binom{n-x}{k-j}.\]
\end{defn}

\begin{defn}
Given a permutation $\sigma\in S_k$ and $n\geq k$, the \emph{$n\times n$ permutation matrix} of $\sigma$ is the matrix $F_\sigma^{\uparrow n}$ such that, for $1\leq x,y\leq n$,
\begin{equation}
\label{eq:fuzzy-defn}
F^{\uparrow n}_\sigma (x,y) = \frac{(n-k)!}{\binom{n}{k}}\sum_{j=1}^{k} f^{\uparrow n}_{k,j}(x) f^{\uparrow n}_{k,\sigma(j)}(y).
\end{equation}
\end{defn}

This definition also extends the usual notion of a permutation matrix.  That is, when $k=n$ we again have $F_\sigma^{\uparrow k}=A_\sigma$ the usual permutation matrix for $\sigma$.  It is natural to think of $F_\sigma^{\uparrow n}$ as a `fuzzy' permutation matrix representing $\sigma$, illustrated in the following example.

\begin{ex}
Let $\sigma=\perm{34251}$.  Matrices $F_\sigma^{\uparrow 5}$  and $F_\sigma^{\uparrow 6}$ are depicted in Figure~\ref{fig:fuzzy-ex}, where intensity of shading is in proportion with entry values, normalized so that line sums agree.
\end{ex}

\begin{figure}[htbp]
\begin{center}
\begin{tikzpicture}[scale=0.5]
\shadebox{5}{3}{100}
\shadebox{4}{4}{100}
\shadebox{3}{2}{100}
\shadebox{2}{5}{100}
\shadebox{1}{1}{100}
\foreach \a in {0,5}
    \draw (\a,0)--(\a,5);
\foreach \a in {0,5}
    \draw (0,\a)--(5,\a);
\end{tikzpicture}
\hspace{3cm}
\begin{tikzpicture}[scale=0.417]
\shadebox{6}{3}{50}
\shadebox{6}{4}{50}
\shadebox{5}{3}{10}
\shadebox{5}{4}{37}
\shadebox{5}{5}{53}
\shadebox{4}{2}{40}
\shadebox{4}{3}{20}
\shadebox{4}{4}{13}
\shadebox{4}{5}{27}
\shadebox{3}{2}{40}
\shadebox{3}{3}{20}
\shadebox{3}{5}{7}
\shadebox{3}{6}{33}
\shadebox{2}{1}{17}
\shadebox{2}{2}{3}
\shadebox{2}{5}{13}
\shadebox{2}{6}{67}
\shadebox{1}{1}{83}
\shadebox{1}{2}{17}
\foreach \a in {0,6}
    \draw (\a,0)--(\a,6);
\foreach \a in {0,6}
    \draw (0,\a)--(6,\a);
\end{tikzpicture}
\caption{A $5 \times 5$ permutation matrix $A_\sigma$ and induced $6 \times 6$ fuzzy variant $F_\sigma^{\uparrow 6}$.}
\label{fig:fuzzy-ex}
\end{center}
\end{figure}
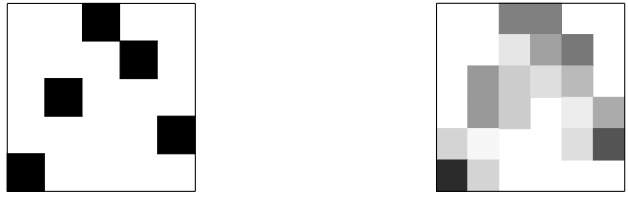

More generally, for a linear combination $\rho=\sum_{i=1}^r c_i\sigma_i$ and $n\geq\max\{|\sigma_i|: c_i\neq0,1\leq i\leq r\}$, we define
\[F_\rho^{\uparrow n}:=\sum_{i=1}^r c_i\cdot F_{\sigma_i}^{\uparrow n}.\]
If $|\sigma_i|=k$ for all $i$ such that $c_i\neq 0$, then we simply write $F_\rho$ for $F_\rho^{\uparrow k}$. 


Next, we make the observation that the cover
matrix $A_\sigma^{\uparrow n}$ and fuzzy permutation
matrix $F_\sigma^{\uparrow n}$ simply differ by a constant matrix.  
The proof is technical, and so it has been relegated to Appendix~B.  Here and later in this section, $J_n$ denotes the $n\times n$ all-ones matrix.

\begin{lem}
\label{lem:FtoA}
If $\sigma\in S_k$ and let $n\geq k$, then
\begin{equation}
\label{eq:FtoA}
A^{\uparrow n}_\sigma = \frac{(n-1)!}{(k-1)!}\left(\frac{1}{k} - \frac{1}{n} \right) J_n + F_\sigma^{\uparrow n}.
\end{equation}
\end{lem}

By linearity, a similar relationship holds for 
$A^{\uparrow n}_\rho$ and $F^{\uparrow n}_\rho$, where $\rho$ is a linear combination.

\subsection{Quasirandom-forcing expressions}

Our main result in this section is the following theorem, which easily implies Theorem~\ref{th:noSmaller}.

\begin{thm}
\label{th:nonZeroCover}
Let $\rho=c_1\sigma_1+\cdots+ c_5\sigma_5$ be a linear combination of permutations and let $n= \max\{|\sigma_i|: c_i\neq 0, 1\leq i\leq 5\}$. If $F_{\rho}^{\uparrow n}$ contains a non-zero entry, then $\rho$ is not quasirandom-forcing.  
\end{thm}

Before diving into the proof of Theorem~\ref{th:nonZeroCover}, let us deduce Theorem~\ref{th:noSmaller} from it. 

\begin{proof}[Proof of Theorem~\ref{th:noSmaller}]
Let $\rho=c_1\sigma_1+\cdots+ c_5\sigma_5$ be a linear combination of permutations such that $c_1,\dots,c_5\geq0$. Our goal is to show that $\rho$ is not quasirandom-forcing. 

First, if $c_1=\cdots=c_5=0$, then $\rho=0$ and every permuton $\mu$ satisfies $d(\rho,\mu)=0$; in particular, $\rho$ is not quasirandom forcing. So, without loss of generality, $c_1>0$. Let  $n= \max\{|\sigma_i|: c_i\neq 0, 1\leq i\leq 5\}$. Note that, for $1\leq i\leq 5$, all of the entries of $F_{\sigma_i}^{\uparrow n}$ are non-negative by the definition of fuzzy permutation matrices. Thus, since all of $c_1,\dots,c_5$ are non-negative, if $F_{\sigma_1}^{\uparrow n}$ contains a non-zero entry, then $F_\rho^{\uparrow n}$ will contain a non-zero entry, and so Theorem~\ref{th:nonZeroCover} will imply that $\rho$ is not quasirandom-forcing. Let $k=|\sigma_1|$. We have
\[F_{\sigma_1}^{\uparrow n}(1,\sigma_1(1))=\frac{(n-k)!}{\binom{n}{k}}\sum_{j=1}^kf^{\uparrow n}_{k,j}(1) f^{\uparrow n}_{k,\sigma(j)}(\sigma(1))\geq \frac{(n-k)!}{\binom{n}{k}}f^{\uparrow n}_{k,1}(1) f^{\uparrow n}_{k,\sigma(1)}(\sigma(1))\]
\[=\frac{(n-k)!}{\binom{n}{k}}\binom{1-1}{1-1} \binom{n-1}{k-1}\binom{\sigma(1)-1}{\sigma(1)-1} \binom{n-\sigma(1)}{k-\sigma(1)}\]
which is positive since $n\geq |\sigma_1|=k$. Thus, $F_{\sigma_1}^{\uparrow n}$ has a positive entry and the proof is complete.
\end{proof}

The rest of the section focuses on the proof of Theorem~\ref{th:nonZeroCover}. It involves repeated applications of the following lemma, which is a generalization of~\cite[Lemma~16]{Chan+20}, proved in exactly the same way. In what follows, $\lambda$ denotes the uniform measure on $[0,1]^2$.

\begin{lem}
\label{lem:LowHigh}
Let $\rho\in \mathcal{A}$. If there exist permutons $\mu_0$ and $\mu_1$ such that
\begin{equation}\label{eq:lowHigh}d(\rho,\mu_0)<d(\rho,\lambda)<d(\rho,\mu_1),\end{equation}
then $\rho$ is not quasirandom-forcing.
\end{lem}

\begin{proof}
For each $z\in (0,1)$, we define a measure $\mu_z$ on $[0,1]^2$ as follows. Given a Borel set $X\subseteq[0,1]^2$, let 
\[X_{0,z}:=\left\{\frac{1}{1-z}(x,y):(x,y)\in X\cap[0,1-z]^2\right\}\]
and
\[X_{1,z}:=\left\{\frac{1}{z}(x-1+z,y-1+z):(x,y)\in X\cap[1-z,1]^2\right\}.\]
Note that $X_{0,z}$ and $X_{1,z}$ are both Borel subsets of $[0,1]^2$. Define $\mu_z(X) := (1-z)\cdot \mu_0(X_{0,z}) + z\cdot\mu_1(X_{1,z})$. Intuitively, $\mu_z$ is obtained by ``scaling down'' the measures $\mu_0$ and $\mu_1$ onto the sets $[0,1-z]^2$ and $[1-z,1]^2$, respectively. The fact that $\mu_z$ is a permuton is easily derived from the fact that $\mu_0$ and $\mu_1$ are both permutons. Now, define $f:[0,1]\to\mathbb{R}$ by $f(z)= d(\rho,\mu_z)$ for all $z\in [0,1]$. The function $f$ is continous and satisfies $f(0)=d(\rho,\mu_0)<d(\rho,\lambda)$ and $f(1)=d(\rho,\mu_1)>d(\rho,\lambda)$; so, by the Intermediate Value Theorem, there exists $z\in (0,1)$ such that $f(z)=d(\rho,\mu_z)=d(\rho,\lambda)$. The permuton $\mu_z$ is clearly not the uniform measure; in particular, it satisfies $\mu_z([1-z,1]\times[0,1-z])=0$. Thus, $\rho$ is not quasirandom-forcing.
\end{proof}

\subsection{Constant covers via gradients}
\label{subsec:gradients}

As a first step toward the proof of Theorem~\ref{th:nonZeroCover}, we reduce to the case where $F_\rho^{\uparrow n}$ is a constant matrix.

\begin{lem}
\label{lem:qr-implies-cc}
For a linear combination $\rho=\sum_{i=1}^rc_i\cdot \sigma_i$ and $n\geq \max\{|\sigma_i|: i=1,\dots,r\}$, if $\rho$ is quasirandom-forcing, then $F_\rho^{\uparrow n}$ is a constant matrix. 
\end{lem}

The proof of Lemma~\ref{lem:qr-implies-cc} borrows several ideas from the recent paper of Kure\v{c}ka~\cite{Kurecka22}, some precursors of which can be found in~\cite[Section~4]{Chan+20}. The overarching aim in the proof is to obtain two permutons $\mu_0$ and $\mu_1$ satisfying \eqref{eq:lowHigh}, which will imply that $\rho$ is not quasirandom-forcing by Lemma~\ref{lem:LowHigh}. The permutons $\mu_0$ and $\mu_1$ will be obtained by performing ``local perturbations'' of the uniform measure $\lambda$. Before presenting the proof, we require several preliminaries. 

Given an $n\times n$ doubly stochastic matrix $M$, i.e. a matrix whose entries are non-negative and row and column sums are equal to one, define $\mu[M]$ to be a permuton obtained in the following way. Divide $[0,1]^2$ into a regular $n\times n$ grid and let $X_{i,j}$ be the cell on the $i$th row and $j$th column of this grid. Now, define $\mu[M]$ in such a way that the restriction of $\mu[M]$ to $X_{i,j}$ is a uniform measure satisfying $\mu[M](X_{i,j})=M(i,j)/n$. Next, for $1\leq i,j\leq n-1$, define $B_{i,j}$ to be the $n\times n$ matrix where 
\[B_{i,j}(a,b)=\begin{cases}1&\text{if }(a,b)\in\{(i,j),(i+1,j+1)\},\\
-1&\text{if }(a,b)\in\{(i+1,j),(i,j+1)\},\\0 &\text{otherwise}.\end{cases}\]
Given a vector $\vec{x}=(x_{i,j}: 1\leq i,j\leq n-1)$ of variables, we define 
\[M_{\vec{x}} = \frac{1}{n}J_n + \sum_{i=1}^{n-1}\sum_{j=1}^{n-1}x_{i,j}\cdot B_{i,j}.\]
Throughout this discussion, we restrict the variables $x_{i,j}$ to be contained in $\left[\frac{-1}{4n},\frac{1}{4n}\right]$. Since each entry of $M_{\vec{x}}$ depends on at most four such variables, we see that all entries of $M_{\vec{x}}$ are non-negative. Also, all of the row and column sums of $M_{\vec{x}}$ are equal to one, and so $M_{\vec{x}}$ is doubly stochastic. For each permutation $\sigma$, integer $n\geq |\sigma|$ and $\vec{x}\in \left[\frac{-1}{4n},\frac{1}{4n}\right]^{(n-1)\times(n-1)}$, define
\[h_{\sigma,n}(\vec{x}):=d(\sigma, \mu[M_{\vec{x}}]).\]
Note that $h_{\sigma,n}(\vec{x})$ is a polynomial in the variables $x_{i,j}$ for $1\leq i,j\leq n-1$.

Note that, if $\vec{x}=\vec{0}$, then $\mu[M_{\vec{x}}]$ is nothing more than the uniform measure $\lambda$. Thus, $h_{\sigma,n}(\vec{0})=\frac{1}{|\sigma|!}$ for every permutation $\sigma$. The proof of Lemma~\ref{lem:qr-implies-cc} involves analyzing the gradient of $h_{\sigma,n}(\vec{x})$ at the point $\vec{x}=\vec{0}$. Fortunately, most of the hard work involved in this analysis was already done for us in~\cite{Kurecka22}. For each permutation $\sigma$ and $\alpha,\beta\in (0,1)$, define
\[P_\sigma(\alpha,\beta):= \lim_{n\to\infty}n^3\frac{\partial}{\partial x_{\lfloor \alpha n\rfloor,\lfloor \beta n\rfloor}}h_{\sigma,n}(\vec{0}).\]
In~\cite[Lemma~5]{Kurecka22}, it is shown that $P_{\sigma}(\alpha,\beta)$ is well defined; i.e., the limit exists for all $\alpha$ and $\beta$. Given $\rho=\sum_{i=1}^rc_i\cdot \sigma_r$ and $n\geq \max\{|\sigma_i|: i=1,\dots,r\}$, define $h_{\rho,n}(\vec{x}):=\sum_{i=1}^rc_i\cdot h_{\sigma_i,n}(\vec{x})$ and $P_\rho(\alpha,\beta):=\sum_{i=1}^r c_i\cdot P_{\sigma_i}(\alpha,\beta)$. To prove Lemma~\ref{lem:qr-implies-cc}, we apply the following lemma.

\begin{lem}[Kure\v{c}ka~{\cite[Lemma~10]{Kurecka22}}]
\label{lem:gradientPoly}
For a non-zero linear combination $\rho=\sum_{i=1}^rc_i\cdot \sigma_i$ where $\sigma_1,\dots,\sigma_r\in S_k$, if $A_\rho$ is not a constant matrix, then there exist $\alpha,\beta\in (0,1)$ such that $P_{\rho}(\alpha,\beta)\neq 0$.
\end{lem}

We are now prepared to prove Lemma~\ref{lem:qr-implies-cc}. 

\begin{proof}[Proof of Lemma~\ref{lem:qr-implies-cc}]
First, consider the case that $\rho=\sum_{i=1}^rc_i\cdot \sigma_i$ where all of $\sigma_1,\dots,\sigma_r$ have the same order, say, $k$. We prove the contrapositive. Suppose that $F_\rho$ is not constant. By Lemma~\ref{lem:FtoA}, we get that $A_\rho$ is not constant either.  By Lemma~\ref{lem:gradientPoly}, there exist $\alpha,\beta\in (0,1)$ such that $P_\rho(\alpha,\beta)\neq 0$. So, by definition of $P_{\sigma}(\alpha,\beta)$, there exists an integer $n$ such that
\[\frac{\partial}{\partial x_{\lfloor \alpha n\rfloor,\lfloor \beta n\rfloor}}h_{\rho,n}(\vec{0}) \neq 0.\]
In particular, the gradient $\nabla h_{\rho,n}(\vec{0})$ is non-zero. Thus, if we let $\vec{x_1} = \varepsilon\cdot\nabla h_{\rho,n}(\vec{0})$ for some $\varepsilon$ very close to zero, then $h_{\rho,n}(\vec{x_1}) > h_{\rho,n}(\vec{0})$; on the other hand, if $\vec{x_0} = - \varepsilon\cdot\nabla h_{\rho,n}(\vec{0})$ for some $\varepsilon$ very close to zero, then $h_{\rho,n}(\vec{x_0}) < h_{\rho,n}(\vec{0})$. The fact that $\rho$ is not quasirandom-forcing now follows from an application of Lemma~\ref{lem:LowHigh} with $\mu_0:=\mu[M_{\vec{x_0}}]$ and $\mu_1:=\mu[M_{\vec{x_1}}]$.

Now, let $\rho=\sum_{i=1}^rc_i\cdot \sigma_i$ where the permutations $\sigma_1,\dots,\sigma_r$ may have different orders and let $n\geq \max\{|\sigma_i|: i=1,\dots,r\}$. Suppose that $\rho$ is quasirandom-forcing. We let $\rho'\in\mathcal{A}$ be defined by
\[\rho':= \sum_{\pi\in S_n}\left(\sum_{i=1}^r c_i\cdot d(\sigma_i,\pi)\right)\pi.\]
Let $\mu$ be a permuton. By linearity of $d(\cdot, \mu)$ and \eqref{eq:projectUp}, we have
\[d(\rho',\mu) = \sum_{\pi\in S_n}\left(\sum_{i=1}^r c_i\cdot d(\sigma_i,\pi)\right)d(\pi,\mu) = \sum_{i=1}^r c_i\cdot \left(\sum_{\pi\in S_n}d(\sigma_i,\pi)d(\pi,\mu)\right)\]
\[=\sum_{i=1}^r c_i\cdot d(\sigma_i,\mu) = d(\rho,\mu).\]
So, since $\rho$ is quasirandom-forcing, the above equality tells us that $\rho'$ is quasirandom forcing as well. Thus, by the result of the previous paragraph, $F_{\rho'}$ and $A_{\rho'}$ are both constant matrices. Now,
\[A_\rho^{\uparrow n} =\sum_{i=1}^r c_i\cdot A_{\sigma_i}^{\uparrow n} =\sum_{i=1}^r c_i\cdot \left(\sum_{\pi\in S_n}d(\sigma_i,\pi)\cdot A_{\pi}\right)=\sum_{\pi\in S_n}\left(\sum_{i=1}^r c_i\cdot d(\sigma_i,\pi)\right)A_\pi =A_{\rho'}.\]
So, $A_\rho^{\uparrow n}$ is a constant matrix as well. By Lemma~\ref{lem:FtoA}, we get that $F_\rho^{\uparrow n}$ is also constant. This completes the proof.
\end{proof}

\subsection{Classification of non-vanishing constant covers}

By Lemma~\ref{lem:qr-implies-cc}, 
we can restrict our attention to expressions $\rho=c_1\sigma_1+\cdots+ c_r\sigma_r$ such that 
\begin{equation}\label{eq:constant-cover} F_\rho^{\uparrow n} =\sum_{i=1}^r c_i  \cdot F_{\sigma_i}^{\uparrow n} = cJ_n\end{equation} 
for $n=\max\{|\sigma_i|: c_i\neq 0,i=1,\dots,r\}$ and some non-zero real number $c$.  The next step of the proof of Theorem~\ref{th:nonZeroCover} argues that, if $\rho$ is a linear combination of only $r\leq 5$ permutations, it is very rare for \eqref{eq:constant-cover} to be satisfied for $c\neq 0$.  In fact, we show that there are only finitely many such choices of $\rho$, up to scaling. To complete the proof of Theorem~\ref{th:nonZeroCover}, it then suffices to find permutons $\mu_0$ and $\mu_1$ satisfying the hypotheses of Lemma~\ref{lem:LowHigh} for each of these finite number of remaining expressions $\rho$. 


For a permutation $\sigma\in S_k$, the sum of the first row of $F_{\sigma}^{\uparrow n}$ is
\[\sum_{y=1}^nF_\sigma^{\uparrow n}(1,y) = \frac{(n-k)!}{\binom{n}{k}}\sum_{y=1}^n\sum_{j=1}^k \binom{1-1}{j-1}\binom{n-1}{k-j}\binom{y-1}{\sigma(j)-1}\binom{n-y}{k-\sigma(j)}\]
\[=\frac{(n-k)!\binom{n-1}{k-1}}{\binom{n}{k}}\sum_{y=1}^n\binom{y-1}{\sigma(1)-1}\binom{n-y}{k-\sigma(1)} = (n-k)!\binom{n-1}{k-1} = \frac{(n-1)!}{(k-1)!}.\]
It follows that the value of $c$ in \eqref{eq:constant-cover} is uniquely determined by
\begin{equation}\label{eq:uniqueC}\sum_{i=1}^r c_i \frac{(n-1)!}{(|\sigma_i|-1)!} = c n.\end{equation}

We also remark that, given permutations $\sigma_1,\dots,\sigma_r$, it is straightforward to check whether they satisfy \eqref{eq:constant-cover} for some coefficients $c_1,\dots,c_r$ and  $c$ as in \eqref{eq:uniqueC}. This is equivalent to expressing $J_n$ as a linear combination of $F_{\sigma_1}^{\uparrow n},\dots,F_{\sigma_r}^{\uparrow n}$ in $\mathbb{R}^{n \times n}$, which can be tested by solving a system of $n^2$ linear equations in $r$ variables.  When $r$ is small compared with $n$, it is a rare occurrence to find solutions, as we see in what follows.

With Lemma~\ref{lem:qr-implies-cc} as motivation, let us say that a linear combination $c_1 \sigma_1 + \dots + c_r \sigma_r$ is a \emph{constant cover} if \eqref{eq:constant-cover} holds, and that it is \emph{non-vanishing} if $c \neq 0$.  For instance, decomposing an $n \times n$ latin square by `level sets' gives a sum of permutation matrices $F_{\sigma_1}+\dots+F_{\sigma_n}=J_n$. Up to scalar multiples, this is the unique class of non-vanishing constant covers using $n$ permutations of length $n$; see also \url{https://oeis.org/A264603}.  Here, we are also interested in constant covers involving permutations of mixed lengths.  

When most permutation lengths are close to $n$, a simple yet effective strategy is to exploit sparsity of the associated matrices.  
For a matrix $A \in \R^{m \times n}$, define \[m_0(A):=|\{(i,j):A_{ij}=0\}|,\] 
the number of zero entries in $A$, and
\[m_*(A):=\max \{m_0(A-cJ_{m,n}): c\in \mathbb{R}\setminus \{0\}\},\]
the maximum number of times a nonzero element is repeated in $A$. Here, $J_{m,n}$ is the $m \times n$ all-ones matrix.  For our classification result, we can get a surprising amount of mileage from the following na\"ive observation.

\begin{lem}
\label{lem:repeats}
If $A+B$ is a nonzero constant matrix, then $m_*(A) \ge m_0(B)$.
\end{lem}

\begin{proof}
Suppose $A+B=cJ_{m,n}$.
If $B_{ij}=0$, then $A_{ij}=c$; hence, element $c$ occurs at least $m_0(B)$ times in $A$.
\end{proof}

Lemma~\ref{lem:repeats} can be applied on a submatrix of each side of \eqref{eq:constant-cover}.  Even the first row contains enough structure to rule out many cases. For this reason, it is useful to gather some basic information about the first row of a fuzzy permutation matrix.

\begin{lem}
\label{lem:first-row}
For $\sigma\in S_k$ and $n\geq k$, the first row of $F_\sigma^{\uparrow n}$ contains exactly $n-k+1$ nonzero elements, and any nonzero element occurs at most twice.
\end{lem}

\begin{proof}
Put $x=1$ in \eqref{eq:fuzzy-defn}, and observe that only the $j=1$ term contributes nonzero summands.  So the first row entries $F_\sigma^{\uparrow n}(1,y)$, $y=1,\dots,n$ are given by
\begin{equation}
\label{eq:first-row}
F_\sigma^{\uparrow n}(1,y) = \frac{k(n-k)!}{n} f_{k,\ell}^{\uparrow n}(y) = \frac{k(n-k)!}{n} \binom{y-1}{\ell-1} \binom{n-y}{k-\ell},
\end{equation}
where $\ell=\sigma(1)$.
This is positive only for those column indices $y$ with 
$\ell \le y \le \ell+n-k$.  This proves the first claim.

Next, the polynomial $f_{k,\ell}^{\uparrow n}(y)$ has degree $k-1$ and exactly $k-1$ distinct real zeros, namely at the integers 
\[y \in [n] \setminus [\ell,\ell+n-k] = \{1,2,\dots,\ell-1,\ell+n-k+1,\dots,n\}.\]  Its derivative has exactly $k-2$ zeros interlacing these.  In particular, over the interval $[\ell,\ell+n-k]$, our polynomial $f_{k,\ell}^{\uparrow n}$ is positive with exactly one local maximum.
The second claim follows from this.
\end{proof}

We now give a consequence for linear combinations of $F_\sigma^{\uparrow n}$.

\begin{lem}
\label{lem:first-row-inequality}
Suppose \eqref{eq:constant-cover} holds for $c\neq0$, and put $k_i:=|\sigma_i|$ for each $i$.  Then
\begin{equation}
\label{eq:degree-bound1}
n-2 \le \sum_{j=1}^{r-1} (n-k_j+1).
\end{equation}
\end{lem}

\begin{proof}
Let $F_i=F_{\sigma_i}^{\uparrow n}$.  
Let $R$ denote the first row of 
$c_1F_1+\dots+c_{r-i}F_{r-1}$.  
By Lemma~\ref{lem:first-row} each term of this sum can contribute at most $n-k_j+1$ nonzero entries.
Therefore, 
$m_0(R) \ge n-\sum_{j=1}^{r-1} n-k_j-1$.
Let $S$ denote the first row of $F_r$.  We have $m_*(S) \le 2$ by the second claim of Lemma~\ref{lem:first-row}.
The result now follows from directly
Lemma~\ref{lem:repeats}.
\end{proof}

Moving from the first row to the full matrix, we could find no simple formula for $m_*(F_\sigma^{\uparrow n})$, as this quantity depends on $\sigma$ in a nontrivial way.  However, when $k$ and $n$ are small, it is simple to count repeated elements using a computer.  Table~\ref{table:single-repeats} displays, for $2 \le k \le n \le 6$, the maximum value of 
$m_*(F_\sigma^{\uparrow n})$ over all $\sigma \in S_k$.

\begin{table}[htbp]
\[
\begin{array}{r|ccccc}
k~ & 2 & 3 & 4 & 5 & 6 \\
\hline
n=4 & 4 & 4 & 4\\
n=5 & 9 & 7 & 4 & 5 \\
n=6 & 4 & 8 & 8 & 5 & 6  \\
\end{array}
\]
\caption{Repeated elements in fuzzy permutation matrices: $\max \{m_*(F_\sigma^{\uparrow n}): |\sigma|=k\}$.}
\label{table:single-repeats}
\end{table}

With a bit more care, we can also analyze repeated elements in a linear combination of two fuzzy permutation matrices.  
Consider the linear combination $c_1 A + c_2 B$ for matrices $A$ and $B$.
An entry is repeated, say at positions $(i,j) \neq (u,v)$
if and only if $c_1(A_{ij}-A_{uv}) = c_2(B_{uv}-B_{ij})$.  So, given $A$ and $B$, we can first identify all candidate coefficient pairs $(c_1,c_2)$, where without loss of generality $c_1=1$ or $(c_1,c_2)=(0,1)$. 
Then, we can compute $m_*(c_1A+c_2B)$ in each case.
We applied this method to find the maximum number of repeated entries in a linear combination of two $6 \times 6$ fuzzy permutation matrices with given permutation lengths.  A few of these values are helpful in what follows.  The results are shown in Table~\ref{table:pair-repeats}.

\begin{table}[htbp]
\[
\begin{array}{r|cccccccccccc}
(k,\ell) & (4,2) & (4,3) & (4,4) & (5,2) & (5,3) & (5,4) & (5,5) \\
\hline
n=6 & 12 & 20 & 12 & 12 & 12 & 16 & 9 \\
\end{array}
\]
\caption{$\max \{m_*(c_1 F_\sigma^{\uparrow n}+c_2 F_\tau^{\uparrow n}) : c_i \in \mathbb{R},~|\sigma|=k,~|\tau|=\ell\}$.}
\label{table:pair-repeats}
\end{table}

Additional structure can be obtained from polynomial degree considerations.
From \eqref{eq:fuzzy-defn}, we see that $F_\sigma^{\uparrow n}(x,y)$ is a polynomial of degree $|\sigma|-1$ in each of $x$ and $y$.  Therefore, if \eqref{eq:constant-cover} holds
then all terms of positive degree must cancel.  For instance, this shows there must exist at least two permutations $\sigma_i$ of maximum length, so that their top degree terms can cancel.  A few generalizations of this observation are given below.

\begin{lem}
\label{lem:degrees}
Suppose that \eqref{eq:constant-cover} holds for $r \in \{4,5\}$ with nonzero 
coefficients $c_i,c$ and distinct $\sigma_i$.  
Put $k_i:=|\sigma_i|$ for each $i$ and assume without 
loss of generality that $n=k_1 \ge k_2 \ge \dots \ge k_r \ge 2$.
Then for each $i=2,\dots,r$, we have
\begin{equation}
\label{eq:degree-bound2}
k_i \ge n+1-\sum_{1\le j<i} (n-k_j+1).
\end{equation}
\end{lem}

\begin{proof}
For each $i=1,\dots,r$, let $f_i(1),\dots,f_i(n)$ be the entries in the first row of 
$F_{\sigma_i}^{\uparrow n}$.  From \eqref{eq:first-row}, we have
\[f_i(x)=\frac{k_i(n-k_i)!}{n}\binom{x-1}{\ell-1} \binom{n-x}{k_i-\ell},\] 
where $\ell=\sigma_i(1)$.  In particular, $f_i(x)$ is a polynomial of degree $k_i-1$.
For $i \ge 2$, put $g_i(x)=c_1 f_1(x)+\dots + c_{i-1}f_{i-1}(x)$.  We consider two cases.

{\sc Case 1}.  $g_i(x)$ is not the zero polynomial.  Let $m_i$ denote the summation on the right side of \eqref{eq:degree-bound2}.  
By Lemma~\ref{lem:first-row}, $g_i(x)$ takes at most $m_i$ nonzero values for $x \in [n]$.
Therefore, $\deg g_i(x) \ge n-m_i$. But $c_1 f_1(x)+\dots+c_r f_r(x)$ is a constant polynomial and the degree of every term at or beyond $f_i(x)$ is at most $k_i-1$.  It follows that $\deg g_i(x) \le k_i-1$.  Putting these together, we get $k_i \ge n+1-m_i$.

{\sc Case 2}.  $g_i(x)$ is the zero polynomial.  Since we have assumed coefficients are nonzero, this can only occur for $3 \le i \le r$.  For $i=r$, since $c_1 f_1+\dots+c_r f_r=g_r+c_rf_r$ is a constant polynomial, we conclude that $f_r$ is a constant polynomial.  This is impossible because $k_r \ge 2$.  Suppose $i=r-1$. Similar reasoning gives that $c_{r-1}f_{r-1}+c_r f_r$ is a nonzero constant polynomial.  Since $\deg(f_j)=k_j-1$, it follows that $k_{r-1}=k_r=k$, say.  Now, observe that the polynomials 
$\binom{x-1}{\ell-1} \binom{n-x}{k-\ell}$, $\ell=1,\dots,k$, form a basis for the
vector space of polynomials in $x$ of degree at most $k-1$.  Moreover, by the Vandermonde identity,
\[\sum_{\ell=1}^k \binom{x-1}{\ell-1} \binom{n-x}{k-\ell} = \binom{n-1}{k-1},\]
a nonzero constant.  Because $c_{r-1}f_{r-1}+c_r f_r$ has a unique representation relative to this basis, we must have $k=2$ and $c_{r-1}=c_r$.  This implies
$\{\sigma_{r-1},\sigma_r\} = \{\perm{12},\perm{21}\}$.  Note that $F_{12}^{\uparrow n}+F_{21}^{\uparrow n}$ is a constant matrix for any $n \ge 2$.  This implies 
$c_1 F_{\sigma_1}^{\uparrow n}+\dots+ c_{r-2}F_{\sigma_{r-2}}^{\uparrow n} = O$,
because it is constant and its first row is zero.  When $r=4$, this is impossible since $\sigma_1\neq \sigma_2$.  Suppose $r=5$ and without loss of generality $c_3=1$. Then $F_{\sigma_3}^{\uparrow n}=-c_1F_{\sigma_1}^{\uparrow n}-c_2F_{\sigma_2}^{\uparrow n}$, a linear combination of two permutation matrices.
If $k_3=n$, this is impossible, since we've assumed $\sigma_1\neq \sigma_2$ and the right side would have two nonzero entries in some row.  If $k_3 \le n-1$, then $n \ge 3$ and every row of $F_{\sigma_3}^{\uparrow n}$ except for the first and last will contain at least three nonzero entries, producing another contradiction.

The only remaining possibility is $i=3$.  The inequality \eqref{eq:degree-bound2} amounts to 
$k_3 \ge n-1$.  To see this, simply pick some row in which the permutation matrices 
$F_{\sigma_1}^{\uparrow n}$ and $F_{\sigma_2}^{\uparrow n}$ disagree.  In the linear combination $c_3 F_{\sigma_3}^{\uparrow n}+c_4 F_{\sigma_4}^{\uparrow n}+c_5F_{\sigma_5}^{\uparrow n}$, the degree of this row as a polynomial is at least $n-2$ but at most $k_3-1$.  Hence, $k_3 \ge n-1$.
\end{proof}

As particular consequences of Lemma~\ref{lem:degrees}, we have $k_1=k_2=n$, $k_3 \ge n-1$, and $k_4 \ge n-3$.  
Results of a similar style, also using polynomial degrees, are found in \cite{Kurecka22}, except that a different family of matrices is analyzed.

We are now ready to undertake the classification of non-vanishing constant covers involving five or fewer permutations.  As a warm-up, it is easy to see that the unique $2$-term constant cover (up to scaling) is $\nu=\perm{12}+\perm{21}$.  It is also not hard to analyze the $3$-term constant covers. Using the notation above, we must have $(k_1,k_2,k_3)=(n,n,n)$ or $(n,n,n-1)$.  Then \eqref{eq:degree-bound1} implies $n \le 4$. The only possibilities are $(k_1,k_2,k_3) \in \{(3,3,3),(3,3,2),(4,4,3)\}$.  The first case corresponds to latin squares.  Up to symmetry and scaling, the only expression giving a constant cover in the second case is 
\begin{equation}
\label{eq:three-term-cc}
\xi=2(\perm{123})-2(\perm{321})-3(\perm{12}),
\end{equation}
which we reference later.  Finally, in the last case, we'd need by Lemma~\ref{lem:repeats} at least $4(4-2)=8$ repeated nonzero entries in $F_{\sigma_3}^{\uparrow 4}$, but this falls far short of the bound reported on the first row of Table~\ref{table:single-repeats}. Note that this characterization together with the result of~\cite{KralPikhurko13,Yanagimoto70} that $S_3$ is not quasirandom-forcing tells us that no non-vanishing constant cover can be quasirandom-forcing.  This proves Theorem~\ref{th:nonZeroCover} for three-term expressions; note that this special case follows from a much stronger theorem of Kure\v{c}ka~\cite{Kurecka22}.

For $r=4$ and $r=5$, more case-work occurs and a computer-assisted search is used.  The computation for $r=4$ is quite easy and succinct, and we state our findings for this case next.

\begin{prop}
\label{prop:constantcover4}
Suppose $c_1 \sigma_1 + \dots + c_4 \sigma_4$ is a non-vanishing constant cover.  Then one of the following cases occur:
\begin{enumerate}
\item[\rm (a)]
all permutation lengths are at most three;
\item[\rm (b)]
the permutations are mutually disjoint of order four (i.e. form a latin square); or
\item[\rm (c)]
up to $D_4$-symmetry and scaling, the linear combination is one of
\begin{align*}
3(\perm{1234}) &+ 3(\perm{4321}) - 4(\perm{123}) + 3(\perm{12}) \\
3(\perm{1234}) &+ 3(\perm{4321}) - 4(\perm{123}) - 3(\perm{21}) \\
3(\perm{1324}) &+ 3(\perm{4231}) - 4(\perm{123}) + 3(\perm{12}) \\
3(\perm{1324}) &+ 3(\perm{4231}) - 4(\perm{123}) - 3(\perm{21}) \\
3(\perm{2143}) &+ 3(\perm{3412}) + 4(\perm{123}) + 3(\perm{12}) \\
3(\perm{2413}) &+ 3(\perm{3142}) + 4(\perm{123}) + 3(\perm{12}) \\
3(\perm{1234}) &+ 3(\perm{4321}) - 2(\perm{123}) -2(\perm{321}) \\
3(\perm{1324}) &+ 3(\perm{4231}) - 2(\perm{123}) -2(\perm{321}) \\
3(\perm{2143}) &+ 3(\perm{3412}) + 2(\perm{123}) +2(\perm{321}) \\
3(\perm{2413}) &+ 3(\perm{3142}) + 2(\perm{123}) +2(\perm{321}) \\
36(\perm{12345}) &-36(\perm{52341}) + 15(\perm{2143}) +10(\perm{321})\\
36(\perm{12345}) &-36(\perm{52341}) - 15(\perm{3412}) -10(\perm{123}) \\
36(\perm{12435}) &-36(\perm{52431}) + 15(\perm{2143}) +10(\perm{321})\\
36(\perm{12435}) &-36(\perm{52431}) - 15(\perm{3412}) -10(\perm{123}).
\end{align*}
\end{enumerate}
\end{prop}

\begin{proof}
Put $k_i=|\sigma_i|$ and $F_i=F_{\sigma_i}^{\uparrow n}$ for each $i=1,\dots,4$.  We know that
$c_1 F_1+\dots +c_4 F_4$ is a nonzero constant matrix. 
Assume without loss of generality that $n=k_1\ge \cdots \ge k_4$.   Since the claim says nothing about permutations of lengths less than four, we may assume $n \ge 4$.

Lemmas~\ref{lem:first-row-inequality} and \ref{lem:degrees} provide bounds on the permutation lengths $k_i$.
In particular, $k_1=k_2=n$ and $k_3 \in \{n,n-1\}$.
Feeding this into \eqref{eq:degree-bound1} gives $n-2 \le 1+1+2$, or $n \le 6$.  The lists $(k_1,\dots,k_4)$ we must consider are of the form $(4,4,-,-)$, $(5,5,4,-)$, $(5,5,5,-)$, or $(6,6,5,-)$.

For the first case, we implemented a computer search over permutations of length at most four, excluding the case of latin squares.  The search returned no linear combinations in which $k_3=4$ and $k_4<4$.  When $k_3=3$, our search produced the first 10 linear combinations listed in (c).

For the second case, we implemented a similar search. 
To counteract the larger space of permutations, we exploited the cancellation of top degree terms contributed by $\sigma_1$ and $\sigma_2$.  That is, we may assume $c_1=1$ and $c_2=-1$, and moreover that rows and columns of $F_1-F_2$ interpolate to polynomials of degree at most $3$.
This search produced the last four linear combinations in the statement.

Consider the third case.  We have $m_0(c_1F_1+c_2F_2+c_3F_3) \ge 5^2-15=10$, since each term has only five nonzero entries.  By Lemma~\ref{lem:repeats}, we would require $m_*(F_4) \ge 10$.  But this contradicts the values in Table~\ref{table:single-repeats}.

In the final case, we similarly have $m_0(c_1F_1+c_2F_2) \ge 6^2-12 =24$. By Lemma~\ref{lem:repeats}, we would need 
$m_*(c_3F_3+c_4F_4) \ge 24$, where $(k_3,k_4) \in \{(5,3),(5,4),(5,5)\}$.  But this contradicts the values in Table~\ref{table:pair-repeats}.
\end{proof}

It is interesting to observe that many of the constant covers listed in (c) of Proposition~\ref{prop:constantcover4} are related to each other via $\nu$, $\xi$, or an expression from (b).

A similar argument, with more case analysis and computer searching, can handle five-term expressions.  Here, we outline this classification, and refer the reader to Appendix C for the full list of covers found.

\begin{prop}
\label{prop:constantcover5}
Suppose $c_1 \sigma_1 + \dots + c_5 \sigma_5$ is a non-vanishing constant cover.
Then $n:=\max\{|\sigma_i|:i=1,\dots,5\} \le 5$.  
Up to scaling and dihedral symmetry, a complete list for $n \in \{4,5\}$ appears in Appendix C.
\end{prop}

\begin{proof}
Put $k_i=|\sigma_i|$ and $F_i=F_{\sigma_i}^{\uparrow n}$ for each $i=1,\dots,5$.  We know that
$c_1 F_1+\dots +c_5 F_5$ is a nonzero constant matrix. 
Assume without loss of generality that $n=k_1\ge \cdots \ge k_5$.

Recall that by Lemma~\ref{lem:degrees}, we have $k_2=n$, $k_3 \in \{n-1,n\}$,
and $k_4 \in \{n-3,\dots,n\}$.  So, Lemma~\ref{lem:first-row-inequality} gives
$n-2 \le 1+1+2+4$, or $n \le 10$.
This reduces the classification to a finite problem. 

For $n \le 5$, a direct search for non-vanishing constant covers was conducted, making use of repeated entries and interpolation to speed up certain cases. This search identified 266 new $5$-term non-vanishing constant covers, 192 of which are induced by latin squares, and 74 of which involve permutations of mixed lengths.  Not included in these are a few constant covers that arise from adding a multiple of $\nu$ or $\xi$ on to one of the 4-term expressions in Proposition~\ref{prop:constantcover4}.
For completeness, the covers are listed in Appendix C along with those for the case $r=4$ in Proposition~\ref{prop:constantcover4}.

A separate computer search showed no non-vanishing constant covers for $6 \le n \le 10$.  An important strategy was used to make this search feasible.
Observe that the first $m$ rows of a fuzzy permutation matrix $F_\sigma^{\uparrow n}$ depend only on $\sigma(1),\dots,\sigma(m)$.  That is, we can rule out a constant cover on the first $m$ rows by searching only over partial permutations of length $m$.  
Our search method takes $(k_1,\dots,k_5)$ as input and begins with $m=1$.  Every list of five partial permutations which produces a constant cover on the first $m$ rows is stored, if any exist.  Then, after incrementing $m$, we search over only those partial permutations whose prefixes of length $m-1$ match one of the previously stored lists.  In practice, this method quickly ruled out each case with $6 \le n \le 10$ using only $m \in \{1,2,3\}$.  In fact, just two cases
\[(k_1,\dots,k_5) \in \{(6,6,6,5,4), (7,7,6,5,4)\}\]
required a search for $m=3$.  

More details of our search are given in Appendix E, along with links to the raw code and to an interface for conducting the search for a given parameter list $(k_1,\dots,k_5)$.
\end{proof}

\subsection{Proof of Theorem~\ref{th:nonZeroCover} and summary of covers}

With the non-vanishing constant cover classification in place, we are now ready to prove our main result of this section.

\begin{proof}[Proof of Theorem~\ref{th:nonZeroCover}]
We know that a quasirandom-forcing linear combination of permutations must involve at least one permutation of length four or more by the result of~\cite{KralPikhurko13,Yanagimoto70} that $S_3$ is not quasirandom-forcing. 
Moreover, from \cite{Chan+20} no latin square of order four induces a quasirandom-forcing linear combination.
 
It remains to check the linear combinations listed in (c) of Proposition~\ref{prop:constantcover4}, and the five-term expressions referenced in Proposition~\ref{prop:constantcover5}. For each such linear combination $\rho$, we computed the $16\times 16$ Hessian matrix of the function $h_{\rho,5}$ defined in Subsection~\ref{subsec:gradients}, as is done in~\cite{Chan+20}. Only five of the expressions failed to produce a Hessian eigenvalue of each sign:
\begin{align*}
\rho_1&=3(\perm{3412})+3(\perm{2143})+2(\perm{321})+2(\perm{123})+t \xi, \\
\rho_2&=-3(\perm{4231})-3(\perm{1324})+2(\perm{321})+2(\perm{123})+t \xi, \\
\rho_3&=6(\perm{4321})+3(\perm{3412})-4(\perm{1324})+\perm{1234}+3(\perm{12}),\\
\rho_4&=-6(\perm{4321})-3(\perm{3412})+4(\perm{1324})-\perm{1234}+3(\perm{21}),\\
\rho_5&=\perm{14253}+\perm{25314}+\perm{31425}+\perm{42531}+\perm{53142},
\end{align*}
where $\xi$ in the first two expressions is given in \eqref{eq:three-term-cc}.  
The first three of these lacked a negative eigenvalue and the last two lacked a positive eigenvalue.
However, for each $\rho_i$, we found a step-permuton $\mu_i$ satisfying $d(\rho_i,\mu_i)<d(\rho_i,\lambda)$ ($i=1,2,3$) or $d(\rho_i,\mu_i)>d(\rho_i,\lambda)$ ($i=4,5$), where $\lambda$ is the uniform permuton.  Thus, these linear combinations fail to be quasirandom forcing by Lemma~\ref{lem:LowHigh}. 
See Appendix C for the Hessian eigenvalues of non-vanishing constant covers, and Appendix D for the five ad hoc step-permutons and density inequalities.
\end{proof}

Table~\ref{table:cc} shows a summary of the non-vanishing constant covers with four or five permutations organized by their list of lengths
$(k_1,k_2,\dots)$. 
The number of distinct covers (up to scaling and dihedral symmetry) is shown, along with the number that our search identified as requiring an ad hoc permuton owing to a positive or negative semidefinite Hessian matrix.

\begin{table}[htbp]
\begin{center}
\begin{tabular}{lcc}
\hline 
$(k_1,k_2,\dots)$ & covers & ad hoc\\
\hline
$(4,4,3,2)$ & $6^\nu$ \\
$(4,4,3,3)$ & $4^\xi$ & 2 \\
$(4,4,4,4)$ & $12^\dagger$  \\
$(5,5,4,3)$ & 4 \\
\hline
$(4,4,3,3,2)$ & 6 \\
$(4,4,3,3,3)$ & 2 \\
$(4,4,4,3,2)$ & 13 \\
$(4,4,4,3,3)$ & 4 \\
\hline
\end{tabular}
\hspace{1cm}
\begin{tabular}{lcc}
\hline 
$(k_1,k_2,\dots)$ & covers & ad hoc\\
\hline
$(4,4,4,4,2)$ & 11 & 2 \\
$(4,4,4,4,3)$ & 9 &  \\
$(5,5,4,3,2)$ & 13 &  \\
$(5,5,4,3,3)$ & 6 &  \\
$(5,5,4,4,2)$ & 7 &  \\
$(5,5,4,4,3)$ & 1 &  \\
$(5,5,4,4,4)$ & 2 &  \\
$(5,5,5,5,5)$ & $192^\dagger$ & 1  \\
\hline
\end{tabular}
\caption{Summary of non-vanishing constant covers; $\dagger$: latin square type covers;\\ $\nu,\xi$: also produce $5$-term covers using 
$\nu=\perm{12}+\perm{21}$ and $\xi=2(\perm{123})-2(\perm{321})-3(\perm{12})$}
\label{table:cc}
\end{center}
\end{table}


\section{Concluding Remarks}
\label{sec:concl}

We have exhibited a quasirandom-forcing linear combination 
\[\rho^*:= \perm{123}+\perm{321}+\perm{2143}+\perm{3412}+ \frac{1}{2}\left(\perm{2413} + \perm{3142}\right)\]
of six permutations.  We also showed that this is best possible, in the sense that no positive linear combination of five or fewer permutations is quasirandom-forcing.  We did not find any other quasirandom-forcing linear combination using only six permutations besides $\rho^*$.  Whether it is unique with these properties is left as a question for future study.

\begin{conj}
If $\rho$ is a quasirandom-forcing linear combination of six permutations, then $\rho=c\cdot \rho^*$ for some $c\in \mathbb{R}$.
\end{conj}

Our argument for ruling out quasirandom-forcing linear combinations of four or five permutations with positive coefficients relied on a classification of non-vanishing constant covers.  To relax the positivity condition, one must also consider `vanishing covers', in which the right side of \eqref{eq:constant-cover} is the zero matrix.
Even with four permutations, this opens up the possibility for infinite families, such as the combination
\begin{equation}
\label{eq:zero-cover}
(\perm{12345...n})+(\perm{21435...n})-(\perm{12435...n})-(\perm{21345...n}).
\end{equation}
Preliminary work suggests that 
\eqref{eq:zero-cover} (and analogous expressions) are possibly the only four or five-term zero covers for $n > 5$.
However, our method of counting repeated nonzero entries in $F_\sigma^{\uparrow n}$ needs to be adapted or replaced in a successful attack on expressions with general coefficients.
\begin{conj}
There are no quasirandom-forcing linear combinations with fewer than six terms, even if negative coefficients are allowed.
\end{conj}
The constant cover problem is interesting in its own right.  A next step would be to expand our classification to expressions involving six (or more) permutations.
\begin{prob}
Classify in greater generality those expressions satisfying \eqref{eq:constant-cover}.
\end{prob}
Finally, for certain applications it may be useful to efficiently compute our density statistic. For example, Even-Zohar and Leng~\cite{EvenZoharLeng21} devised a $\tilde{O}(|\pi|)$-time algorithm for computing $d(\tau^*,\pi)$ for a permutation $\pi$, where $\tau^*$ is the Bergsma--Dassios statistic defined in the introduction. This provides an efficient method for testing whether two $[0,1]$-valued random variables $X$ and $Y$ with continuous joint cdf are independent. In fact, the results of~\cite{EvenZoharLeng21} can be used to efficiently compute $d(\rho,\pi)$ for a large class of expressions $\rho\in\mathcal{A}$, which they refer to as \emph{corner trees}. As it turns out, $\rho^*$ is not a corner tree, and so the results of~\cite{EvenZoharLeng21} are insufficient for computing $d(\rho^*,\pi)$ in time $\tilde{O}(|\pi|)$; however, using other algorithms in~\cite{EvenZoharLeng21}, one can do much better than the trivial bound of $O(n^4)$. It is an interesting problem to develop a faster algorithm for computing $d(\rho^*,\pi)$ than the one that can be derived from~\cite{EvenZoharLeng21}. 



\bibliographystyle{amsplain}

\begin{thebibliography}{99}

\bibitem{BaloghClemenLidicky22}
J.~Balogh, F.~C. Clemen, and B.~Lidick\'{y}.
\newblock Solving {T}ur\'{a}n's tetrahedron problem for the {$\ell_2$}-norm.
\newblock {\em J. Lond. Math. Soc. (2)}, 106(1):60--84, 2022.

\bibitem{Balogh+15}
J.~Balogh, P.~Hu, B.~Lidick\'{y}, O.~Pikhurko, B.~Udvari, and J.~Volec.
\newblock Minimum number of monotone subsequences of length 4 in permutations.
\newblock {\em Combin. Probab. Comput.}, 24(4):658--679, 2015.

\bibitem{BergsmaDassios14}
W.~Bergsma and A.~Dassios.
\newblock A consistent test of independence based on a sign covariance related
  to {K}endall's tau.
\newblock {\em Bernoulli}, 20(2):1006--1028, 2014.

\bibitem{BucicLongShapiraSudakov21}
M.~Buci\'{c}, E.~Long, A.~Shapira, and B.~Sudakov.
\newblock Tournament quasirandomness from local counting.
\newblock {\em Combinatorica}, 41(2):175--208, 2021.

\bibitem{Chan+20}
T.~F.~N. Chan, D.~Kr\'{a}\v{l}, J.~A. Noel, Y.~Pehova, M.~Sharifzadeh, and
  J.~Volec.
\newblock Characterization of quasirandom permutations by a pattern sum.
\newblock {\em Random Structures Algorithms}, 57(4):920--939, 2020.

\bibitem{ChungGraham90}
F.~R.~K. Chung and R.~L. Graham.
\newblock Quasi-random hypergraphs.
\newblock {\em Random Structures Algorithms}, 1(1):105--124, 1990.

\bibitem{ChungGraham91tourn}
F.~R.~K. Chung and R.~L. Graham.
\newblock Quasi-random tournaments.
\newblock {\em J. Graph Theory}, 15(2):173--198, 1991.

\bibitem{ChungGrahamWilson89}
F.~R.~K. Chung, R.~L. Graham, and R.~M. Wilson.
\newblock Quasi-random graphs.
\newblock {\em Combinatorica}, 9(4):345--362, 1989.

\bibitem{Cooper04}
J.~N. Cooper.
\newblock Quasirandom permutations.
\newblock {\em J. Combin. Theory Ser. A}, 106(1):123--143, 2004.

\bibitem{Cooper+22}
J.~W. Cooper, D.~Kr\'{a}\v{l}, A.~Lamaison, and S.~Mohr.
\newblock Quasirandom {L}atin squares.
\newblock {\em Random Structures Algorithms}, 61(2):298--308, 2022.

\bibitem{CoreglianoParenteSato19}
L.~N. Coregliano, R.~F. Parente, and C.~M. Sato.
\newblock On the maximum density of fixed strongly connected subtournaments.
\newblock {\em Electron. J. Combin.}, 26(1):Paper No. 1.44, 48, 2019.

\bibitem{CoreglianoRazborov17}
L~N. Coregliano and A.~A. Razborov.
\newblock On the density of transitive tournaments.
\newblock {\em J. Graph Theory}, 85(1):12--21, 2017.

\bibitem{EvenZoharLeng21}
C.~Even-Zohar and C.~Leng.
\newblock Counting small permutation patterns.
\newblock In {\em Proceedings of the 2021 {ACM}-{SIAM} {S}ymposium on
  {D}iscrete {A}lgorithms ({SODA})}, pages 2288--2302. [Society for Industrial
  and Applied Mathematics (SIAM)], Philadelphia, PA, 2021.

\bibitem{Garbe+19}
F.~Garbe, R.~Hancock, J.~Hladk\'{y}, and M.~Sharifzadeh.
\newblock Theory of limits of sequences of {L}atin squares.
\newblock {\em Acta Math. Univ. Comenian. (N.S.)}, 88(3):709--716, 2019.

\bibitem{Gowers08}
W.~T. Gowers.
\newblock Quasirandom groups.
\newblock {\em Combin. Probab. Comput.}, 17(3):363--387, 2008.

\bibitem{Griffiths13}
S.~Griffiths.
\newblock Quasi-random oriented graphs.
\newblock {\em J. Graph Theory}, 74(2):198--209, 2013.

\bibitem{Hancock+23}
R.~Hancock, A.~Kabela, D.~Kr\'{a}\v{l}, T.~Martins, R.~Parente, F.~Skerman, and
  J.~Volec.
\newblock No additional tournaments are quasirandom-forcing.
\newblock {\em European J. Combin.}, 108:Paper No. 103632, 10, 2023.

\bibitem{HavilandThomason89}
J.~Haviland and A.~Thomason.
\newblock Pseudo-random hypergraphs.
\newblock volume~75, pages 255--278. 1989.
\newblock Graph theory and combinatorics (Cambridge, 1988).

\bibitem{HladkyKralNorin17}
J.~Hladk\'{y}, D.~Kr\'{a}\v{l}, and S.~Norin.
\newblock Counting flags in triangle-free digraphs.
\newblock {\em Combinatorica}, 37(1):49--76, 2017.

\bibitem{Hoeffding48}
W.~Hoeffding.
\newblock A non-parametric test of independence.
\newblock {\em Ann. Math. Statistics}, 19:546--557, 1948.

\bibitem{Hoppen+13}
C.~Hoppen, Y.~Kohayakawa, C.~Gustavo Moreira, B.~R\'{a}th, and
  R.~Menezes~Sampaio.
\newblock Limits of permutation sequences.
\newblock {\em J. Combin. Theory Ser. B}, 103(1):93--113, 2013.

\bibitem{KalyanasundaramShapira13}
S.~Kalyanasundaram and A.~Shapira.
\newblock A note on even cycles and quasirandom tournaments.
\newblock {\em J. Graph Theory}, 73(3):260--266, 2013.

\bibitem{Kendall38}
M.~G. Kendall.
\newblock A new measure of rank correlation.
\newblock {\em Biometrika}, 30:81--93, 1938.

\bibitem{KohayakawaRodlSkokan02}
Y.~Kohayakawa, V.~R\"{o}dl, and J.~Skokan.
\newblock Hypergraphs, quasi-randomness, and conditions for regularity.
\newblock {\em J. Combin. Theory Ser. A}, 97(2):307--352, 2002.

\bibitem{KralPikhurko13}
D.~Kr\'{a}\v{l} and O.~Pikhurko.
\newblock Quasirandom permutations are characterized by 4-point densities.
\newblock {\em Geom. Funct. Anal.}, 23(2):570--579, 2013.

\bibitem{Kurecka22}
M.~Kure\v{c}ka.
\newblock Lower bound on the size of a quasirandom forcing set of permutations.
\newblock {\em Combin. Probab. Comput.}, 31(2):304--319, 2022.

\bibitem{Lovasz12}
L.~Lov\'{a}sz.
\newblock {\em Large networks and graph limits}, volume~60 of {\em American
  Mathematical Society Colloquium Publications}.
\newblock American Mathematical Society, Providence, RI, 2012.

\bibitem{Razborov07}
A.~A. Razborov.
\newblock Flag algebras.
\newblock {\em J. Symbolic Logic}, 72(4):1239--1282, 2007.

\bibitem{Razborov08}
A.~A. Razborov.
\newblock On the minimal density of triangles in graphs.
\newblock {\em Combin. Probab. Comput.}, 17(4):603--618, 2008.

\bibitem{Razborov13}
A.~A. Razborov.
\newblock Flag algebras: an interim report.
\newblock In {\em The mathematics of {P}aul {E}rd\H{o}s. {II}}, pages 207--232.
  Springer, New York, 2013.

\bibitem{Rodl86}
V.~R\"{o}dl.
\newblock On universality of graphs with uniformly distributed edges.
\newblock {\em Discrete Math.}, 59(1-2):125--134, 1986.

\bibitem{SliacanStromquist17}
J.~Slia\v{c}an and W.~Stromquist.
\newblock Improving bounds on packing densities of 4-point permutations.
\newblock {\em Discrete Math. Theor. Comput. Sci.}, 19(2):Paper No. 3, 18,
  2017.

\bibitem{Spearman04}
C.~Spearman.
\newblock The proof and measurement of association between two things.
\newblock {\em J. Psychol.}, 15:72--101, 1904.

\bibitem{Thomason87}
A.~Thomason.
\newblock Pseudorandom graphs.
\newblock In {\em Random graphs '85 ({P}ozna\'{n}, 1985)}, volume 144 of {\em
  North-Holland Math. Stud.}, pages 307--331. North-Holland, Amsterdam, 1987.

\bibitem{Yanagimoto70}
T.~Yanagimoto.
\newblock On measures of association and a related problem.
\newblock {\em Ann. Inst. Stat. Math.}, 22:57--63, 1970.

\bibitem{Zhang18+}
E.~Zhang.
\newblock On quasirandom permutations.
\newblock Slides available at
  \url{https://math.mit.edu/research/highschool/primes/materials/2018/conf/9-2\%20Zhang.pdf},
  Presented at the MIT PRIMES Conference in May 2018.

\end{thebibliography}


\begin{dajauthors}
\begin{authorinfo}[gabriel]
  Gabriel Crudele\\
  McGill University\\
  Montr\'eal, Canada\\
  gabriel\imagedot{}crudele\imageat{}mail\imagedot{}mcgill\imagedot{}ca \\
\end{authorinfo}
\begin{authorinfo}[peter]
  Peter Dukes\\
  University of Victoria\\
  Victoria, Canada\\
  dukes\imageat{}uvic\imagedot{}ca \\
  \url{https://web.uvic.ca/~dukes/}
\end{authorinfo}
\begin{authorinfo}[jon]
  Jonathan A. Noel\\
  University of Victoria\\
  Victoria, Canada\\
  noelj\imageat{}uvic\imagedot{}ca \\
  \url{https://jonathannoel.ca/}
\end{authorinfo}
\end{dajauthors}

\end{document}